\documentclass[oneside,reqno,11pt]{amsart}   \usepackage{graphicx,amsfonts,amssymb,amsmath,amsthm,url,amscd} \usepackage{enumerate} \usepackage{color} \usepackage{chngcntr}
\usepackage[usenames,dvipsnames]{xcolor} \usepackage[normalem]{ulem} \usepackage{comment}
\usepackage{graphicx, amsmath, amssymb, amsfonts, amsthm, tikz, amscd,mathtools}
\usetikzlibrary{matrix,arrows,decorations.pathmorphing}
\usepackage[a4paper, margin=2.7cm]{geometry} \usepackage{verbatim}
\usepackage{calc}
\usepackage{dsfont}

\usepackage{graphicx, amsmath, amssymb, amsfonts, amsthm, stmaryrd, tikz, amscd} \usepackage[all]{xy} \usetikzlibrary{matrix,arrows,decorations.pathmorphing}

\usepackage[hyperfootnotes=false, colorlinks, citecolor= RoyalBlue, urlcolor=blue, linkcolor=blue ]{hyperref}

\makeatletter
\let\orgdescriptionlabel\descriptionlabel
\renewcommand*{\descriptionlabel}[1]{%
  \let\orglabel\label
  \let\label\@gobble
  \phantomsection
  \protected@edef\@currentlabel{#1}%
  \let\label\orglabel
  \orgdescriptionlabel{#1}%
}
\def\paragraph{
	\@startsection{paragraph}{4}
	\z@{.5\linespacing\@plus.7\linespacing}{-.5em}%
	{\normalfont\itshape}}
\makeatother

\usepackage[shortlabels]{enumitem}

\newtheorem{thm}{Theorem}
\theoremstyle{plain}

\newtheorem{conjecture}{Conjecture}
\newtheorem{cor}[thm]{Corollary}

\newtheorem{lem}[thm]{Lemma}

\newtheorem{prop}[thm]{Proposition}

\theoremstyle{remark}
\newtheorem{rem}[thm]{Remark}
\newtheorem*{acknowledgements}{Acknowledgements}

\newcommand{\vect}[1]{{\boldsymbol{#1}}}

\newcommand{\arccosh}{\mathop{\rm arccosh}\nolimits}

\renewcommand{\mod}{\mathop{\rm mod}\nolimits}
\newcommand{\sign}{\mathop{\rm sign}\nolimits}

\renewcommand{\div}{\, | \,}

\newcommand{\notdiv}{\mathopen{\mathchoice
             {\not{|}\,}
             {\not{|}\,}
             {\!\not{\:|}}
             {\not{|}}
             }}

\newcommand{\im}{\mathop{{\rm Im}}\nolimits}

\newcommand{\GL}[1]{\mathop{\rm GL}_{#1} \nolimits}
\newcommand{\SL}[1]{\mathop{\rm SL}_{#1} \nolimits}

\newcommand{\PGL}[1]{\mathop{\rm PGL}_{#1} \nolimits}

\newcommand{\SO}[1]{\mathop{\rm SO}_{#1} \nolimits}

\newcommand{\Mat}[2]{\mathop{\rm Mat}_{#1 \times #2} \nolimits}

\DeclareMathOperator{\JL}{JL}

\newcommand{\tr}{\mathop{\rm tr}\nolimits}

\renewcommand{\Im}{\mathop{\rm Im}\nolimits}

\renewcommand{\hat}{\widehat}

\newcommand{\quotient}[2]{
        \mathchoice
            {
                \text{\raise1ex\hbox{$#1$}\Big/\lower1ex\hbox{$#2$}}%
            }
            {
                #1\,/\,#2
            }
            {
                #1\,/\,#2
            }
            {
                #1\,/\,#2
            }
    }
		
\newcommand{\lquotient}[2]{
        \mathchoice
            {
                \text{\lower1ex\hbox{$#1$}\Big \backslash \raise01ex\hbox{$#2$}}%
            }
            {
                #1\,\backslash\,#2
            }
            {
                #1\,\backslash\,#2
            }
            {
                #1\,\backslash\,#2
            }
    }

\newcommand{\rquotient}[2]{
        \mathchoice
            {
                \text{\raise01ex\hbox{$#1$}\Big/\lower1ex\hbox{$#2$}}%
            }
						{
                #1\,/\,#2
            }
            {
                #1\,/\,#2
            }
            {
                #1\,/\,#2
            }
    }
		
\newcommand{\lrquotient}[3]{
        \mathchoice
            {
                \text{\lower1ex\hbox{$#1$}\Big \backslash \raise01ex\hbox{$#2$}\Big/\lower1ex\hbox{$#3$}}%
            }
            {
                #1\,\backslash\,#2\,/\,#3
            }
            {
                #1\,\backslash\,#2\,/\,#3
            }
            {
                #1\,\backslash\,#2\,/\,#3
            }
    }

\newcommand{\sm}{\left(\begin{smallmatrix}}
\newcommand{\esm}{\end{smallmatrix}\right)}
\newcommand{\bpm}{\begin{pmatrix}}
\newcommand{\ebpm}{\end{pmatrix}}

\DeclareFontFamily{U}{matha}{\hyphenchar\font45}
\DeclareFontShape{U}{matha}{m}{n}{
	<5> <6> <7> <8> <9> <10> gen * matha
	<10.95> matha10 <12> <14.4> <17.28> <20.74> <24.88> matha12
}{}
\DeclareSymbolFont{matha}{U}{matha}{m}{n}
\DeclareFontFamily{U}{mathx}{\hyphenchar\font45}
\DeclareFontShape{U}{mathx}{m}{n}{
	<5> <6> <7> <8> <9> <10>
	<10.95> <12> <14.4> <17.28> <20.74> <24.88>
	mathx10
}{}
\DeclareSymbolFont{mathx}{U}{mathx}{m}{n}

\DeclareMathSymbol{\obot}         {2}{matha}{"6B}
\DeclareMathSymbol{\bigobot}       {1}{mathx}{"CB}

\numberwithin{equation}{section}
\begin{document}

\bibliographystyle{alpha}

\title{Theta functions, fourth moments of eigenforms, and the sup-norm problem III}

\author{Raphael S. Steiner}
\address{Computing Systems Lab, Huawei Zurich Research Center, Thurgauerstrasse 80, 8050 Zurich, CH}%
\email{raphael.steiner.academic@gmail.com}%



\dedicatory{Dedicated to Professor Peter Sarnak on the occasion of his seventieth birthday.}

\begin{abstract} In the prequel, a sharp bound in the level aspect on the fourth moment of Hecke--Maa{\ss} forms with an inexplicit (in fact exponential) dependency on the eigenvalue was given. In this paper, we develop further the framework of explicit theta test functions in order to capture the eigenvalue more precisely. We use this to reduce a sharp hybrid fourth moment bound to an 
	intricate counting problem. Unconditionally, we give a hybrid bound, which is sharp in the level aspect and with a slightly larger than convex dependency on the eigenvalue.
\end{abstract}
\maketitle


\section{Introduction}

Let $\varphi$ be an $L^2$-normalised eigenfunction of the Laplace--Beltrami operator $-\Delta$ (Maa{\ss} form) on a hyperbolic surface $\Gamma \backslash \mathbb{H}$ equipped with an invariant probability measure. A classical question in harmonic analysis or quantum chaos is the study of $L^p$-norms of $\varphi$. Here, we consider the $L^{\infty}$-norm of $\varphi$ also known as the sup-norm problem. For general compact surfaces, one has the bound (see \cite{SeegerSogge}):
\begin{equation} \label{eq:conv-eigenvalue}
	\| \varphi \|_{\infty} \le C_{\Gamma} (1+|\lambda_{\varphi}|)^{\frac{1}{4}} \| \varphi \|_2,
\end{equation}
for some constant $C_{\Gamma}$ depending on $\Gamma$ and where $\lambda_{\varphi}$ denotes the $(-\Delta)$-eigenvalue of $\varphi$.
Here, $\lambda_{\varphi}$ denotes the $(-\Delta)$-eigenvalue of $\varphi$. In this generality, the bound is sharp and equality is attained for the $2$-dimensional sphere for example. For hyperbolic surfaces, however, it is believed that this bound is far from sharp and that it should hold with an exponent strictly smaller than $\frac{1}{4}$ \cite{IS95}. Without further asumptions beyond constant negative curvature, only a factor of $\log(2+|\lambda_{\varphi}|)$ has been saved over the bound \eqref{eq:conv-eigenvalue}, which is a result due to B\'erard \cite{Berard77}.

In an arithmetic setting, i.e.\@ the underlying lattice $\Gamma$ is arithmetic and the form $\varphi$ is a (new) Hecke--Maa{\ss} eigenform, more can be said. In a breakthrough paper, Iwaniec and Sarnak \cite{IS95} showed
\begin{equation} \label{eq:IS95}
	\| \varphi \|_{\infty} \le C_{\Gamma,\epsilon} (1+|\lambda_{\varphi}|)^{\frac{5}{24}+\epsilon} \| \varphi \|_2,
\end{equation}
for any $\epsilon > 0$ and a constant $C_{\Gamma, \epsilon}$ depending only on $\Gamma$ and $\epsilon$.
Their amplification technique has been widely adapted since then. The family of hyperbolic surfaces $\Gamma_0(N) \backslash \mathbb{H}$, where $\Gamma_0(N) = \{ \sm a & b \\ c & d \esm | c \equiv 0 \mod(N)  \}$,  in particular, has received much attention. Here, the interest lies not only in the dependence on $\lambda_{\varphi}$, but also the dependence on $N$, respectively the co-volume of $\Gamma_0(N)$, which is referred to as the level aspect. Assume from now on that $\Gamma_0(N)\backslash \mathbb{H}$, respectively $\Gamma \backslash \mathbb{H}$, is equipped with the invariant probability measure. Nowadays standard techniques (see for example \cite{BH10}) show that one may take $C_{\Gamma_0(N)}= C_{\epsilon} N^{\frac{1}{2}+\epsilon}$ for any $\epsilon >0$ in \eqref{eq:conv-eigenvalue}. This is commonly referred to as the convex, local, or sometimes ``trivial'' bound. In the case where $N$ is not square-free, the convex bound may be even further decreased (see \cite{MarshallConvex}). In a pioneering paper, Blomer and Holowinsky \cite{BH10} managed to prove the first sub-convex bound in the level aspect for $N$ square-free. Namely, they proved
\begin{equation}\begin{aligned}  \label{eq:BH}
	\| \varphi \|_{\infty} &\ll_{\epsilon} N^{\frac{1}{2}-\frac{1}{37} + \epsilon} (1+|\lambda_{\varphi}|)^{\frac{11}{4} + \epsilon} \| \varphi \|_2 \text{  and} \\
	\| \varphi \|_{\infty} &\ll_{\epsilon} N^{\frac{1}{2} + \epsilon} (1+|\lambda_{\varphi}|)^{\frac{5}{24} + \epsilon} \| \varphi \|_2,	
\end{aligned}\end{equation}
where we've introduced Vinogradov's notation, see \S\ref{sec:vinogradov-notation}.
Subsequent improvements in the square-free level aspect have been made by Helfgott--Ricotta\footnote{unpublished.}, Templier \cite{HT1}, Harcos--Templier \cite{HT2,HT2}, culminating in the hybrid bound by Templier \cite{Thybrid}:
\begin{equation} \label{eq:TempHybrid}
	\| \varphi \|_{\infty} \ll_{\epsilon} N^{\frac{1}{3} + \epsilon} (1+|\lambda_{\varphi}|)^{\frac{5}{24} + \epsilon} \| \varphi \|_2.
\end{equation}
For for a more detailed discussion on the level aspect and sub-convex bounds in the non-square-free level aspect we refer to \cite{Saha14, Saha20, SahaHu20} and references therein/-to. From now on, we assume the level $N$ to be square-free.

Recently, Khayutin, Nelson, and the author \cite{SupThetaII}, building on earlier work \cite{SupThetaI, S3sup}, have tackled the fourth moment in the level aspect and were able to bound it sharply. As a consequence, they proved for any $\Lambda >0$ and $\varphi$ with $|\lambda_{\varphi}| \le \Lambda$ that
\begin{equation}
	\label{eq:KNS}
	\|\varphi\|_{\infty} \ll_{\Lambda,\epsilon} N^{\frac{1}{4}+\epsilon},
\end{equation}
The dependence of the implied constant on $\Lambda$, respectively $\lambda_{\varphi}$, in \eqref{eq:KNS} is in fact of exponential nature and thus far from the bounds \eqref{eq:BH} and \eqref{eq:TempHybrid}. In this sequel, we remedy this and give a polynomial dependence on $\lambda_{\varphi}$.

\begin{thm} \label{thm:hybrid-indiv} Let $N$ be a square-free integer and $\varphi$ a Hecke--Maa{\ss} newform on $\Gamma_0(N) \backslash \mathbb{H}$ with $(-\Delta)$-eigenvalue $\lambda_{\varphi}$. Then,
	$$
	\| \varphi \|_{\infty} \ll_{\epsilon} N^{\frac{1}{4}+\epsilon} (1+|\lambda_{\varphi}|)^{\frac{3}{8}+\epsilon} \|\varphi\|_2.
	$$
\end{thm}
The method of proof follows the prequel \cite{SupThetaII}, in which the theta correspondence and the exceptional isomorphism ${\rm PGSO}(2,2) \cong {\rm PGL}_2 \times {\rm PGL}_2$ is used to derive a geometric expression for a fourth moment leading to a second moment matrix counting problem. The particular choice of theta function plays a crucial role in the weighting of the Hecke--Maa{\ss} forms $\varphi$. In the prequel \cite{SupThetaII}, a crude test function with everywhere positive weighting was chosen at the infinite place. This allowed for a simpler treatment of the reduction to a counting problem as well as a simpler (but highly non-trivial) counting problem. This, however, came at the expense of an exponential dependency on the spectral paramenter. Here, we extend the family of permitted test functions at the infinite place 
to allow for more tailored test functions which, in turn, enable us to capture the eigenvalue aspect more precisely. Unconditionally, we are able to prove:

\begin{thm} \label{thm:fourth-hybrid unconditional} Let $\Lambda > 1$ and $z,w \in \mathbb{H}$. Further, let $N$ be a square-free integer and $\{\varphi\}$ an orthonormal set of Hecke--Maa{\ss} newforms on $\Gamma_0(N) \backslash \mathbb{H}$ with $(-\Delta)$-eigenvalues $\{\lambda_{\varphi}\}$ bounded by $\Lambda$. Then,
	$$
	\sum_{\substack{\varphi \\ |\lambda_{\varphi}| \le \Lambda}} \left( |\varphi(z)|^2 - |\varphi(w)|^2 \right)^2 \ll_{\epsilon} N^{1+\epsilon} \Lambda^{\frac{3}{2}+\epsilon} \left(1+\frac{N}{\Lambda^{\frac{1}{2}}} \left( H(z)^2+H(w)^2 \right) \right),
	$$
	where
	\begin{equation} \label{eq:intro-H}
	H(z) = \max_{\gamma \in A_0(N)} \Im(\gamma z) 
	\end{equation}
	and the minimum is taken over all Atkin--Lehner operators $A_0(N)$ of level $N$.
\end{thm}
Whilst the exponent $\frac{3}{2}$ in the eigenvalue is larger than the conjectured and optimal exponent of $1$ (see \cite[Conj.\@ 1.1.]{ChamCorrSumsOfSquares} for an explicit variant of this conjecture), it marks a polynomial dependence that may prove useful in certain applications. The reason for the larger than optimal dependence on the eigenvalue is our inability to bound the resulting counting problem adequately. Consequently, we also establish a conditional version of Theorem \ref{thm:fourth-hybrid unconditional} with the optimal dependency on the eigenvalue that is contingent on adequate bounds on second moment matrix count which we will discuss in more detail in Section \ref{sec:lattice-counting}.

\begin{thm} \label{thm:fourth-hybrid-conditional} Let $\Lambda > 1$ and $z,w \in \mathbb{H}$. Further, let $N$ be a square-free integer and $\{\varphi\}$ an orthonormal set of Hecke--Maa{\ss} newforms on $\Gamma_0(N) \backslash \mathbb{H}$ with $(-\Delta)$-eigenvalues $\{\lambda_{\varphi}\}$ bounded by $\Lambda$. Assume the counting conjecture, Conjecture \ref{conj:heart-lattice-counting}, then
	$$
	\sum_{\substack{\varphi \\ |\lambda_{\varphi}| \le \Lambda}} \left( |\varphi(z)|^2 - |\varphi(w)|^2 \right)^2 \ll_{\epsilon} N^{1+\epsilon} \Lambda^{1+\epsilon} \left(1+ N \left( H(z)^2+H(w)^2 \right) \right),
	$$
	where $H(z)$ is as in \eqref{eq:intro-H}.
\end{thm}
We also establish a version of Theorems \ref{thm:hybrid-indiv}-\ref{thm:fourth-hybrid-conditional} for a more general class of arithmetic lattices.

\begin{thm}
	\label{thm:Eichler-order-main-thm}
	Let $R$ be an Eichler order of square-free level\footnote{See \S\ref{sec:notation} for the meaning of level in this context.} in a non-split indefinite quaternion algebra $B$ over $\mathbb{Q}$. Let $\Gamma$ be the (reduced) norm one elements inside $\PGL{2}(\mathbb{R}) \cong {\rm PB}^{\times}(\mathbb{R})$, which is a lattice. Let $V$ denote the co-volume of $\Gamma$. Let $\Lambda > 1$ and $z,w \in \mathbb{H}$. Further, let $\{\varphi\}$ an orthonormal set of Hecke--Maa{\ss} newforms on $\Gamma \backslash \mathbb{H}$ with $(-\Delta)$-eigenvalues $\{\lambda_{\varphi}\}$ bounded by $\Lambda$. Then,
	\begin{equation} \label{eq:4th-moment-thm-Eichler}
	\sum_{\substack{\varphi \\ |\lambda_{\varphi}| \le \Lambda}} \left( |\varphi(z)|^2 - |\varphi(w)|^2 \right)^2 \ll_{\epsilon} V^{1+\epsilon} \Lambda^{\frac{3}{2}+\epsilon}.
	\end{equation}
	Consequently, for an individual Hecke--Maa{\ss} newform $\varphi$ on $\Gamma\backslash\mathbb{H}$, we have
	\begin{equation} \label{eq:indv-thm-Eichler}
	\| \varphi \|_{\infty} \ll_{\epsilon} V^{\frac{1}{4}+\epsilon} (1+|\lambda_{\varphi}|)^{\frac{3}{8}}  \|\varphi\|_2.
	\end{equation}
	Assuming the counting conjecture, Conjecture \ref{conj:heart-lattice-counting}, then the exponents $\frac{3}{2}$ of $\Lambda$ in  \eqref{eq:4th-moment-thm-Eichler}, respectively $\frac{3}{8}$ of $(1+|\lambda_{\varphi}|)$ in \eqref{eq:indv-thm-Eichler}, may be replaced by $2$, respectively $\frac{1}{4}$.
\end{thm}

\subsection{The archimedean test function}

Let $B$ be an indefinite quaternion algebra over $\mathbb{Q}$ such that the completion at the infinite place $B(\mathbb{R})$ is isomorphic to the matrix algebra $\Mat{2}{2}(\mathbb{R})$. In the prequel \cite[Appendix A]{SupThetaII}, the relevant technical machinery was developped for a Schwartz function $\Phi_{\infty}$ on $\Mat{2}{2}(\mathbb{R})$ that is bi-$\SO{2}(\mathbb{R})$-invariant and satisfies a certain partial differential equation \eqref{eq:isotypical-PDE}. There, the choice 
$$
\Phi_{\infty}(\vect{x}) = e^{-\pi \|\vect{x}\|_{2}^2}
$$
was made whose ``Selberg/Harish-Chandra" transform
\begin{equation} \label{eq:intro-selberg-transform}
\int_{\SL{2}(\mathbb{R})} \tau \Phi_{\infty}(\vect{x}\gamma) \Xi_{\frac{1}{2}+it}(\gamma) d\gamma =  2\sqrt{\tau} K_{it}(2\pi \tau)
\end{equation}
matches up with the Whittaker function at $\sm \tau^{1/2} & \\ & \tau^{-1/2} \esm$ of an arithmetically normalised Hecke--Maa{\ss} newform with eigenvalue $\frac{1}{4}+t^2$. Here, $\vect{x} \in \Mat{2}{2}(\mathbb{R})$ is any matrix of determinant $\tau > 0$, $\Xi_s$ is the spherical function for the principal series on $\PGL{2}(\mathbb{R})$, and the measure $d\gamma = \frac{dxdy}{y^2} \frac{d\theta}{2\pi}$ for $\gamma = \sm 1 & x \\ & 1 \esm \sm y^{1/2} & \\ & y^{-1/2} \esm \sm \cos(\theta) & \sin(\theta) \\ -\sin(\theta) & \cos(\theta) \esm $. In the end, it is the exponential decay of the $L^2$-norm of an arithmetically normalised Hecke--Maa{\ss} newform in the spectral parameter $t$ that is responsible for the exponential loss in the spectral parameter in the prequel \cite{SupThetaII}.

In this paper, we invert the ``Selberg/Harish-Chandra" transform \eqref{eq:intro-selberg-transform}. For suitable test functions $h(t;\tau)$, the inversion is given by
\begin{equation} \label{eq:intro-inversion-Selberg-transform}
	\Phi_{\infty}(\vect{x}) = 	\frac{1}{4\pi |\det(\vect{x})|} \int_{-\infty}^{\infty} h(t;\det(\vect{x})) \Xi_{\frac{1}{2}+it}\left(  \frac{\|\vect{x}\|_2^2-2|\det(\vect{x})|}{4|\det(\vect{x})|} \right) \tanh(\pi t) t dt,
\end{equation}
when $\det(\vect{x}) \neq 0$. By construction, this function is bi-$\SO{2}(\mathbb{R})$-invariant. In Section \ref{sec:arch-test-func}, we further show decay properties and that it satisfies the required partial differential equation wherever it is smooth. Unfortunately, we cannot guarantee that the function is everywhere smooth. To work around this issue, we prove, via an application of Green's Theorem, that the function is nevertheless in the $L^2$-closure of the space of bi-$\SO{2}(\mathbb{R})$-invariant Schwartz functions who satisfy the required partial differential equation. This arguement closely follows the one given in the prequel \cite[\S 6]{SupThetaI}, where the required test function (an extended Bergman kernel) was not smooth either.

\subsubsection{Long window} For the long eigenvalue window, we chose the test function $\Phi_{\infty}$ that arises from 
$$
h(t;\tau) = \cosh\left( \tfrac{\pi}{2}-\tfrac{1}{T}  \right) \cdot 2 \sqrt{|\tau|} K_{it}(2 \pi |\tau|) 
$$
in the aforementioned construction \eqref{eq:intro-inversion-Selberg-transform}. This particular choice of test function comes with several qualities that we exploit in this paper. The extra exponential factor compared to \eqref{eq:intro-selberg-transform} compensates fully for (the exponential) loss that is introduced via the $L^2$-norm of an arithmetically normalised Hecke--Maa{\ss} newform at least for spectral parameter $t \le T$. Furthermore, the Fourier transform of $h(t;\tau)$ in $t$ is elementary, rapidly decreasing, and the action of the diagonal subgroup of $\SL{2}(\mathbb{R})$ via the Weil representation on $\Phi_{\infty}$ is oscillatory. The former allows us to write $\Phi_{\infty}$ as an Abel transform of an elementary function and the latter is exploited for further cancellation which is crucial in the proof of Theorem \ref{thm:fourth-hybrid-conditional} and its analogue Theorem \ref{thm:Eichler-order-main-thm}.

\subsubsection{Short window} In the prequel \cite{SupThetaII}, a hybrid bound which is convex in the eigenvalue has been derived for an individual form for the analogous problem on definite quaternion algebras. This has been achieved via a non-sharp bound for the fourth moment over a short eigenvalue window. Naturally, the question arises whether the same is possible for indefinite quaternion algebras via analogous means. 
Unfortunatly, we were so far unable to overcome the analytical difficulties that arise in this pursuit. In particular, we were led to analyse the integral
$$
\int_{-\infty}^{\infty} e^{-\frac{1}{2}S^2(r \pm \frac{1}{T}i-\ell)^2} e^{i ( \beta\sinh(\ell) \pm T((r-\ell)) )} d\ell,
$$
to which we were unable to apply Debye's method of steepest descent in sufficient uniformity in the parameters $T$, $S$, $\beta$, and $r$, where $S$ is a bit bigger than $1$, $0 < \beta \ll T$, and $r \in \mathbb{R}$. We further note that traditional stationary phase arguments (such as in \cite[\S8]{BlomerKhanYoung} for a singularity of degree two) run into trouble as the magnitudes of the derivatives of the phase may oscillate with the order of the derivative, thus leading to insufficient bounds.

Perhaps, these analytic difficulties simplify if the conditions on the test function $\Phi_{\infty}$ are further relaxed. For example, such that it must no longer satisfy the partial differential equation \eqref{eq:spherical-PDE}, in which case the theta lift from $\mathop{\rm PB}^{\times}  \times \mathop{\rm PB}^{\times}$ to $\PGL{2}$ is no longer spherical, or if one drops the bi-$\SO{2}(\mathbb{R})$-invariance, in which case one also picks up Maa{\ss} forms of different weights.



\subsection{Overview}

In Section \ref{sec:lattice-counting}, we introduce an intricate counting problem, which when bounded sufficiently strongly (see Conjecture \ref{conj:heart-lattice-counting}) yields sharp bounds for the fourth moments, Theorems \ref{thm:fourth-hybrid-conditional}, \ref{thm:Eichler-order-main-thm}. In the same section, we also give bounds on some lattice counting problems in the split case that were not required in the prequel.

In Section \ref{sec:arch-test-func}, we prove an extension of the allowed test functions for the theta correspondence used in this paper. The latter is briefly recalled in Section \ref{sec:theta}.

In Section \ref{sec:longwindow}, we prove the reduction of the fourth moment to a counting problem and deduce the main theorems.

In the appendix, we collect some bounds on special functions which are required in parts of the paper.

\begin{acknowledgements} I am indebted to Paul Nelson, Ilya Khayutin for endless discussions on this and other topics, and to Peter Sarnak for his continued support and encouragement over the years. I would also like to thank Peter Humphries for valuable conversations.
	
This work began at the Institute for Advanced Study, where I was supported by the National Science Foundation Grant No. DMS – 1638352 and the Giorgio and Elena Petronio Fellowship Fund II, and completed at the Institute for Mathematical Research (FIM) at ETH Z\"urich.
	
\end{acknowledgements}

\section{Notation} \label{sec:notation} We keep the notation in line with the prequel \cite{SupThetaII} whenever possible.

Let $B$ be an indefinite quaternion algebra over $\mathbb{Q}$ with reduced discriminant $d_B$, i.e.\@ the product of all finite primes at which $B$ is non-split. Let $R \subset B$ be an Eichler order of level $N$ with $N$ being square-free and coprime to $d_B$. In other words, an order $R$ such that for every prime $p \div d_B$, $R \otimes_{\mathbb{Z}} \mathbb{Z}_p$ is a maximal order in $B \otimes_{\mathbb{Q}} \mathbb{Q}_p$, and for every prime $p \notdiv d_B$, $R \otimes_{\mathbb{Z}} \mathbb{Z}_p \cong \sm \mathbb{Z}_p & \mathbb{Z}_p \\ N\mathbb{Z}_p & \mathbb{Z}_p \esm$. We fix, once and for all, an identification $B_{\infty} = B \otimes_{\mathbb{Q}} \mathbb{R} \cong \Mat{2}{2}(\mathbb{R})$. Then, the group of units of reduced norm one $\Gamma = R^1 \hookrightarrow B^1 \otimes_{\mathbb{Q}} \mathbb{R} \cong \SL{2}(\mathbb{R})$ gives rise to a Fuchsian group of the first kind acting properly discontinuous on the upper-half plane $\mathbb{H} = \SL{2}(\mathbb{R}) \slash K_{\infty}$, where $K_{\infty}= \SO{2}(\mathbb{R})$. We shall equip $\mathbb{H}$ with the measure $\frac{dxdy}{y^2}$, where $z=x+iy \in \mathbb{H}$, and denote by $V_{d_B,N} = (d_BN)^{1+o(1)}$ the co-volume of $\Gamma$. We equip $\Gamma \backslash \mathbb{H}$ with the corresponding invariant probability measure.

The Laplace--Beltrami operator $\Delta$ acts (by extention) on $L^2(\Gamma \backslash \mathbb{H})$ and we denote by $\mathcal{F}$ an orthonormal basis of Hecke--Maa{\ss} forms of the cuspidal newspace of $L^2(\Gamma \backslash \mathbb{H})$. That is, a function $\varphi \in \mathcal{F}$ satisfies
\begin{itemize}
	\item $-\Delta \varphi = \lambda_{\varphi} \varphi$, $\lambda_{\varphi} = \frac{1}{4}+t_{\varphi}^2$,
	\item $\varphi$ is an eigenfunction of almost all Hecke operators,
	\item $\varphi$ is orthogonal to all $\varphi' \in L^2(\Gamma' \backslash \mathbb{H}) \subset L^{2}(\Gamma \backslash \mathbb{H})$ where $\Gamma'$ comes from a super order $R' \supsetneq R$,
	\item $\varphi$ decays exponentially at every cusp of $\Gamma \backslash \mathbb{H}$.
\end{itemize}

%
%
%

\subsection{The split case} In the case where $B $ is split, we may and shall assume that $B= \Mat{2}{2}(\mathbb{Q})$, $B_{\infty} = \Mat{2}{2}(\mathbb{R})$, and $R= \sm \mathbb{Z} & \mathbb{Z} \\ N \mathbb{Z}  & \mathbb{Z} \esm$, such that $R^1 = \Gamma_0(N)$. We further define a height function on $\Gamma_0(N) \backslash \mathbb{H}$:
$$
H(z) = \max_{\gamma \in A_0(N)} \im(\gamma z),
$$
where $A_0(N)< \GL{2}(\mathbb{Q})^+ $ denotes the subgroup generated by $\Gamma_0(N)$ and all the Atkin--Lehner operators. For $g \in \SL{2}(\mathbb{R})$, we also let $H(g) :=H(gi)$.

\subsection{Partially dualised lattices}

For a divisor $\ell \div d_BN$, we denote by $R(\ell)$ the lattice in $B$ whose local components $R(\ell) \otimes_{\mathbb{Z}} \mathbb{Z}_p$ are given by
\begin{itemize}
	\item the dual lattice $(R \otimes_{\mathbb{Z}} \mathbb{Z}_p)^{\vee}$ if $p \div \ell$, and
	\item $R \otimes_{\mathbb{Z}} \mathbb{Z}_p$ otherwise.
\end{itemize}

\subsection{Conjugation}

For $g \in \SL{2}(\mathbb{R})$, we let
$$
R(\ell;g):= g^{-1} R(\ell)g.
$$

\subsection{Tailored coordinates} 

On $\Mat{2}{2}(\mathbb{R})$ (and thus on $B_{\infty}$ via identification), we introduce the coordinate system
$$
[a,b,c]+d = \begin{pmatrix}
	d+c & b+a \\ b-a & d-c
\end{pmatrix}.
$$
In these coordinates, we have
$$
\tr([a,b,c]+d) = 2d, \quad \det([a,b,c]+d)= a^2-b^2-c^2+d^2.
$$
We further introduce the following notations for $\gamma=[a,b,c]+d\in \Mat{2}{2}(\mathbb{R})$:
\begin{equation} \begin{gathered}
		P(\gamma)= a^2+b^2+c^2+d^2, \quad \tau(\gamma) = \det(\gamma) = a^2-b^2-c^2+d^2,\\
		\diamondsuit(\gamma) = P(\gamma)^2-\tau(\gamma)^2 =4(a^2+d^2)(b^2+c^2).
		\label{eq:P-tau-diam-def}
\end{gathered} \end{equation}

%
%

\subsection{Archimedean regions} For $L >0$ and $\delta \in (0,1]$ we denote by $\Omega(\delta,L)$ the set of all elements $[a,b,c]+d \in \Mat{2}{2}(\mathbb{R})$ for which 
$$
a^2+b^2+c^2+d^2 \le L^2, \quad b^2+c^2 \le \delta L^2.
$$
With $\Omega^{\star}(\delta,L)$ we denote the non-zero elements. Similarly, we denote by $\Psi(\delta,L)$ the set of all elements $[a,b,c]+d \in \Mat{2}{2}(\mathbb{R})$ for which 
$$
a^2+b^2+c^2+d^2 \le L^2, \quad a^2+d^2 \le \delta L^2,
$$
and with $\Psi^{\star}(\delta,L)$ we denote the non-zero elements.

\subsection{Vinogradov's notation} \label{sec:vinogradov-notation}

For two expressions $A,B$, we write $A \ll B$ if there is an absolute constant $C > 0$ such that $|A| \le CB$. If the constant $C$ is not absolute, but depends on a range of parameters, then those parameters will be indicated as a subscript of $\ll$. We also let $A \gg B \Leftrightarrow B \ll A$ and
$
A \asymp B \Leftrightarrow \left(A \ll B \wedge B \ll A \right)
$
with the same convention on parameters as subscripts.

We reserve $\epsilon$ to be a special parameter that is sufficiently (depending on the context) and arbitrarily small but positive quantity that may vary from line to line.

\section{Lattice counting} \label{sec:lattice-counting}

%
%
%
%

\subsection{Asymptotic notation}
For the lattice counting, we introduce the following asymptotic relation:
$$
A \prec B \quad \Leftrightarrow \quad A \ll_{\epsilon} (d_BN(1+L))^{\epsilon} B.
$$

\subsection{Counting propositions and conjectures}
Here, we collect the counting propositions and conjectures required to prove Theorems \ref{thm:hybrid-indiv}-\ref{thm:Eichler-order-main-thm}.

\begin{prop}[{\cite[Lemma 3.6]{SupThetaII}}] \label{prop:Omega-lattice-counting}
Let $g \in \SL{2}(\mathbb{R})$. 
We have
	\begin{multline*}
		\sum_{\substack{\gamma_1,\gamma_2 \in R(\ell;g) \cap \Omega^{\star}(\delta,L) \\ \det(\gamma_1)=\det(\gamma_2)}} 1 \prec \ell L^2 \left(1+\ell^{\frac{1}{2}}H(g) \delta^{\frac{1}{2}}+\ell^{\frac{1}{2}}H(g)\delta L + \frac{\ell^2}{d_BN} \delta L^2 \right) \\
		\times \left(1+\ell^{\frac{1}{2}}H(g) \delta^{\frac{1}{2}}+\ell^{\frac{1}{2}}H(g)\delta^{\frac{1}{2}} L + \frac{\ell^2}{d_BN} \delta L^2 \right).
	\end{multline*}
Terms involving $H(g)$ are to be omitted if $B$ is non-split ($\Leftrightarrow d_B > 1$). Moreover, the sum is empty if $L \ll \min\{\ell^{-\frac{1}{2}},\ell^{-1} H(g)^{-1}\delta^{-\frac{1}{2}}\}$.
\end{prop}

\begin{prop}[{\cite[Lemma 3.7]{SupThetaII}} in the case of non-split $B$]
	\label{prop:Psi-lattice-counting}
Let $g\in \SL{2}(\mathbb{R})$. We have 
	\begin{equation*}
		\sum_{\substack{\gamma_1,\gamma_2 \in R(\ell;g) \cap \Psi^{\star}(\delta,L) \\ \det(\gamma_1)=\det(\gamma_2)}} 1 \prec \ell L^2 \Biggl(1+ \ell^{\frac{1}{2}}H(g) \delta^{\frac{1}{2}} +  \frac{\ell}{(d_BN)^{\frac{1}{2}}}L
		+\ell H(g)\delta^{\frac{1}{2}} L  + \frac{\ell^2}{d_BN} \delta^{\frac{1}{2}} L^2 \Biggr)^2.
	\end{equation*}
	Terms involving $H(g)$ are to be omitted if $B$ is non-split ($\Leftrightarrow d_B > 1$). Moreover, the sum is empty if $L \ll \min\{\ell^{-\frac{1}{2}},\ell^{-1} H(g)^{-1}\delta^{-\frac{1}{2}}\}$.
\end{prop}

\begin{conjecture} \label{conj:heart-lattice-counting} Let $g\in \SL{2}(\mathbb{R})$. We have
	\begin{multline*}
		\sum_{\substack{\gamma_1,\gamma_2 \in R(\ell;g)\\ \gamma_1,\gamma_2 \in \Omega^\star(\delta,L) \cup \Psi^\star(\delta,L) \\ \det(\gamma_1)=\det(\gamma_2) \\ |\diamondsuit(\gamma_1)-\diamondsuit(\gamma_2)| \le  \heartsuit L^4}} 1 \prec \ell L^2 \Biggl(1  +\frac{\ell^2}{d_BN} \delta^{\frac{1}{2}}L^2 + \frac{\ell^4}{(d_BN)^2} \min\{\heartsuit, \delta\} L^4 \\
		+ \ell H(g)^2 + \ell^{\frac{3}{2}} H(g)^2 \delta^{\frac{1}{4}} L+\ell^2 H(g)^2 \delta^{\frac{1}{2}}L^2 \Biggr), 
	\end{multline*}
	where the terms involving $H(g)$ are to be omitted if $B$ is non-split ($\Leftrightarrow d_B > 1$).
\end{conjecture}

The two propositions were proven (at least in part) in the prequel \cite[Lemmas 3.6 \& 3.7]{SupThetaII}. The part that is missing is the statement of Proposition \ref{prop:Psi-lattice-counting} for $B$ being the matrix algebra. We give a proof of that case in Section \ref{sec:psi-counting-proof}. Jointly, these two estimates suffice to prove the unconditional Theorems \ref{thm:hybrid-indiv},\ref{thm:fourth-hybrid unconditional} and the unconditional part of Theorem \ref{thm:Eichler-order-main-thm}. They are, however, insufficient to prove the sharp upper bound for the fourth moment. A more stringent counting of pairs of matrices is required. To this avail, we've formulated the smallest upper bound required (by the method of this paper) as Conjecture \ref{conj:heart-lattice-counting}.

The main difficulty in the new restriction $|\diamondsuit(\gamma_1)-\diamondsuit(\gamma_2)| \le \heartsuit L^4$ lies in the fact that it involves both $\gamma_1$ and $\gamma_2$. In the prequels of this work, the strategy of separation was to split the matrices $\gamma_1, \gamma_2$ into their trace and traceless parts and apply a divisor bound to the trace part. This proved highly effective for the condition $\det(\gamma_1) = \det(\gamma_2)$ as it came with little to no cost when $\ell = 1$ and the cost when $\ell > 1$ could be mitigated by weighting these counts less by a suitable choice of Siegel domains. This new condition, however, resists such and other separations insofar as all our attempts have introduced losses that could not be compensated for.

\subsection{Proof of Proposition \ref{prop:Psi-lattice-counting}} \label{sec:psi-counting-proof} For the proof, we require a couple of lemmas.

\begin{lem}[{\cite[Lemma 1]{HT2}}\footnote{The reference gives the slightly weaker bound obtained by omitting the factor $(c,N)$, but the stronger bound that we have stated follows from their proof, keeping track of $(c,N)$ at each step rather than bounding it from below by $1$.}] \label{lem:spacing}
	Let $z\in \mathbb{H}$ with maximal imaginary part under the orbit of the Atkin--Lehner operators $A_0(N)$ of $\Gamma_0(N)$ with $N$ squarefree. Then, we have
	$$
	\Im(z) \ge \frac{\sqrt{3}}{2N} \quad \text{ and } \quad |cz+d|^2 \ge \frac{(c,N)}{N}
	$$
	for any $(c,d) \in \mathbb{Z}^2$ distinct from $(0,0)$.
\end{lem}

%
%

\begin{lem}[{\cite[Lemma 6.3]{SupThetaII}}] \label{lem:latticepointcount}
	Let $\mathcal{K}\subseteq \mathbb{R}^n$ be a closed convex $0$-symmetric set of positive volume. Let $\Lambda\subset \mathbb{R}^n$ be a lattice and let $ \lambda_1 \le \lambda_2 \le \dots \le \lambda_n$ denote the successive minima of $\mathcal{K}$ on $\Lambda$. Then
		$$
		|\mathcal{K} \cap \Lambda| \asymp_n \prod_{i=1}^n \left( 1+\frac{1}{\lambda_i} \right).
		$$
	\end{lem}

	\begin{lem}[{\cite[Lemma 2.1]{HT3}}] \label{lem:latticeball}
		Let $\Lambda \subset \mathbb{R}^2$ be a lattice of rank $2$ and $B\subseteq \mathbb{R}^2$ a ball of radius $R$ (not necessarily centred at $0$). If $\lambda_1 \le \lambda_2$ are the successive minima of $\Lambda$, then
		$$
		|B \cap \Lambda| \ll 1+\frac{R}{\lambda_1}+\frac{R^2}{\lambda_1\lambda_2}.
		$$	
	\end{lem}

\begin{prop} \label{prop:psi-count-trace-free-lattice} Let $g \in \SL{2}(\mathbb{R})$, $B=\Mat{2}{2}(\mathbb{Q})$ the matrix algebra, and $R = \sm \mathbb{Z} & \mathbb{Z} \\ N \mathbb{Z} & \mathbb{Z} \esm$. The first successive minima of $R(\ell;g)^{0}$ with respect to $\Psi(\delta,1)$ is $\gg \min\{\ell^{-\frac{1}{2}}, \ell^{-1}H(g)^{-1}\delta^{-\frac{1}{2}}\}$ and
	\begin{equation*}
	\left| R(\ell;g)^{0} \cap \Psi(\delta,L) \right|  \prec 1+\left( \ell^{\frac{1}{2}}+\ell H(g) \delta^{\frac{1}{2}} \right)L+\left(\frac{\ell^{\frac{3}{2}}}{N^{\frac{1}{2}}}+\ell^{\frac{3}{2}} y \delta^{\frac{1}{2}}\right)L^2+ \frac{\ell^2}{N} \delta^{\frac{1}{2}}L^3.
	\end{equation*}
\end{prop}

\begin{proof} Since $R(\ell)$ is invariant under the Atkin--Lehner operators $A_0(N)$, we may assume that $g i = z = x+iy$ with $y$ maximal under the orbit of $A_0(N)$ such that $H(g)=y$. Furthermore, since $\Psi^\star(\delta,L)$ and $\det(\cdot)$ are bi-$K_{\infty}$-invariant, we may also suppose that $g$ is upper-triangular, i.e.\@ $g= \sm 1 & x \\  & 1 \esm \sm y^{1/2} & \\ & y^{-1/2} \esm $. Let $\gamma = \sm a & b \\ c & -a \esm \in R(\ell)^0$.
	Thus, $a \in \mathbb{Z}$, $b \in \frac{1}{\ell}\mathbb{Z}$, and $c \in \frac{N}{\ell} \mathbb{Z}$. From
	$$
	g^{-1} \gamma g = \begin{pmatrix} a-cx  &  \tfrac{1}{y}(2ax+b-cx^2)  \\  cy & cx-a 	\end{pmatrix} \in \Psi(\delta,L),
	$$
	it follows that
	\begin{equation}
		\left| a-cx \right|  \ll L, \quad |cy| \ll L, \label{eq:psi-a}
		\end{equation}
		\begin{equation}
		\left| 2ax+b-c(x^2+y^2) \right|  \ll y\delta^{\frac{1}{2}}L, \label{eq:psi-mixed}
		\end{equation}
Note that \eqref{eq:psi-a} implies
	\begin{equation}
		\left|cz-a\right| \ll L.  
	\end{equation}
	Suppose $\gamma \neq 0$. Then, either $(a,c) \neq (0,0)$ and $L \gg |cz-a| \ge \ell^{-\frac{1}{2}}$ by Lemma \ref{lem:spacing} or $a=c=0$ and $y\delta^{\frac{1}{2}}L \gg |b| > \ell^{-1}$. Hence, the first successive minima $\lambda_1$ of $R(\ell;g)^0$ with respect to $\Psi(\delta,1)$ satisfies
	\begin{equation} \label{eq:psi-succ-est-first}
	\lambda_1 \gg \min\{\ell^{-\frac{1}{2}}, \ell^{-1}y^{-1}\delta^{-\frac{1}{2}}\}.
	\end{equation}
	Furthermore, we may upper-bound
	\begin{equation} \label{eq:psi-upper-bnd}
	\left| R(\ell;g)^{0} \cap \Psi(\delta,L) \right| \ll \left( 1+ \ell^{\frac{1}{2}}L + \frac{\ell}{Ny} L^2 \right) \left(1+\ell y\delta^{\frac{1}{2}}L \right),
	\end{equation}
	where the first factor is an upper bound for the number of values of $(a,c)$ by Lemmas \ref{lem:spacing}, \ref{lem:latticeball} and the second factor is an upper bound for the number of values of $b$ once $a$ and $c$ are fixed which is derived from \eqref{eq:psi-mixed}.
	
	Alternatively, we see that \eqref{eq:psi-a} and \eqref{eq:psi-mixed} lead to
	\begin{equation}
		2y\left|cx-a\right| \ll yL, \quad \left| c(x^2-y^2) -2ax-b \right| \ll yL
	\end{equation}
	and further
	\begin{equation} \label{eq:Omega-HT-ineq}
		\left| c z^2 -2az -b \right| \ll yL.
	\end{equation}
	We may now bound the number of values of $c$ through \eqref{eq:psi-a} and subsequently, once $c$ is fixed, the number of values of $(a,b)$ by means of Lemmas \ref{lem:spacing}, \ref{lem:latticeball} and \eqref{eq:Omega-HT-ineq}. This yields
	\begin{equation} \label{eq:psi-upper-bnd-delta=1}
		\left| R(\ell;g)^{0} \cap \Psi(\delta,L) \right| \ll \left( 1+\frac{\ell L}{Ny} \right) \left(1+ \ell^{\frac{1}{2}} N^{\frac{1}{2}} y L + \ell y L^2\right).
	\end{equation}
	We note that this last argument is precisely the one given in \cite{HT3}, which may be seen by increasing $\delta$ to $1$ and noting $\Psi(1,L)= \Omega(1,L)$, cf.\@ \cite[\S 9]{SupThetaII}.

	By Lemma \ref{lem:latticepointcount}, we know that
	\begin{equation} \label{eq:psi-exact-count}
		\left| R(\ell;g)^{0} \cap \Psi(\delta,L) \right| = \left| R(\ell;g)^{0} \cap L\Psi(\delta,1) \right| \asymp 1+\frac{L}{\lambda_1}+\frac{L^2}{\lambda_1\lambda_2} + \frac{L^3}{\lambda_1\lambda_2\lambda_3},
	\end{equation}
	where $\lambda_1 \le \lambda_2 \le \lambda_3$ are the successive minima of $R(\ell;g)^0$ with respect to $\Psi(\delta,1)$. By comparing \eqref{eq:psi-upper-bnd} with \eqref{eq:psi-exact-count}, we find
	\begin{equation} \label{eq:psi-succ-est}
	\frac{1}{\lambda_1\lambda_2\lambda_3} \ll \frac{\ell^2 \delta^{\frac{1}{2}}}{N}, \quad \text{and } \quad \frac{1}{\lambda_1\lambda_2} \ll \frac{\ell^{\frac{3}{2}}}{N^{\frac{1}{2}}} \delta^{\frac{1}{4}}+\frac{\ell}{Ny} + \ell^{\frac{3}{2}}y \delta^{\frac{1}{2}}
	\end{equation}
	from taking $L \to \infty$, respectively $L= \ell^{-\frac{1}{2}} N^{\frac{1}{2}}\delta^{-\frac{1}{4}}$. Similarly, we may compare \eqref{eq:psi-upper-bnd-delta=1} with \eqref{eq:psi-exact-count} to arrive at
	\begin{equation}\label{eq:psi-succ-est2}
	\frac{1}{\lambda_1\lambda_2} \ll \frac{\ell^{\frac{3}{2}}}{N^{\frac{1}{2}}}+\ell y
	\end{equation}
	by taking $L= \ell^{-\frac{1}{2}}N^{\frac{1}{2}}$ and using $y \gg N$ which follows from Lemma \ref{lem:spacing}. By combining \eqref{eq:psi-succ-est} and \eqref{eq:psi-succ-est2}, we find
	\begin{equation} \label{eq:psi-succ-est-2comb}
		\frac{1}{\lambda_1\lambda_2} \ll \frac{\ell^{\frac{3}{2}}}{N^{\frac{1}{2}}}+ \min\left\{ \ell y, \frac{\ell}{N y} \right\} + \ell^{\frac{3}{2}}y \delta^{\frac{1}{2}} \ll \frac{\ell^{\frac{3}{2}}}{N^{\frac{1}{2}}} + \ell^{\frac{3}{2}}y \delta^{\frac{1}{2}}.  
	\end{equation}
	We conclude the Lemma from \eqref{eq:psi-exact-count} and the estimates on the successive minima \eqref{eq:psi-succ-est-first}, \eqref{eq:psi-succ-est}, and \eqref{eq:psi-succ-est-2comb}.
\end{proof}

\begin{proof}[Proof of Proposition \ref{prop:Psi-lattice-counting}, split case]
	The argument follows the one given in the prequels \cite[\S9]{SupThetaI}, \cite[\S3.3 \& \S3.4]{SupThetaII}. Each matrix $\gamma \in R(\ell;g)$ is split into a multiple of the identity and a traceless matrix $\gamma^0$:
	$$
	R(\ell; g) \ni \gamma = \tfrac{1}{2} \tr(\gamma)+\gamma^0 \in \tfrac{1}{2}\mathbb{Z} \oplus \tfrac{1}{2}R(\ell;g)^0.
	$$
	The equality of determinants
	$$
	\tfrac{1}{4}\tr(\gamma_1)^2 + \det(\gamma_1^0) = \det(\gamma_1) = \det(\gamma_2) = \tfrac{1}{4} \tr(\gamma_2)^2+ \det(\gamma_2^0)
	$$
	and a divisor bound on the traces is then used to derive
	\begin{multline} \label{eq:psi-reduct-to-ternary}
		\sum_{\substack{\gamma_1,\gamma_2 \in R(\ell;g) \cap \Psi^{\star}(\delta,L) \\ \det(\gamma_1) = \det(\gamma_2) }} 1 \prec \left| R(\ell; g)^0 \cap \Psi(\delta, 2L) \right| \\
		\times \left( \left| R(\ell; g)^0 \cap \Psi(\delta, 2L) \right| + \delta^{\frac{1}{2}} L  \max_{\substack{n \in \frac{1}{n} \mathbb{Z} \\ |n| \le 4 L^2 }} \left| R(\ell; g)^0 \cap \Omega(1, 2L) \cap \det{}^{-1}(\{n\})  \right| \right),
	\end{multline}
	see \cite[Eqs.\@ (3.10) \& (3.11)]{SupThetaII} for the detailled argument. Now, the first successive minima of $R(\ell;g)$ with respect to $\Psi(\delta,1)$ is at least the minimum out of $\delta^{-\frac{1}{2}}$ and $K := \min\{\ell^{-\frac{1}{2}}, \ell^{-1} H(g)^{-1} \delta^{-\frac{1}{2}}\}$ by Proposition \ref{prop:psi-count-trace-free-lattice}. Note $\delta^{-\frac{1}{2}} \ge 1 \ge \ell^{-\frac{1}{2}} \ge K$. Thus, the set over which is summed is empty if $L$ is smaller than that as zero is excluded in $\Psi^\star(\delta, L)$. Hence, the sum is $0$ in that case and the statement holds trivially. For $L \gg K$, we have by Proposition \ref{prop:psi-count-trace-free-lattice} that
	\begin{equation} \label{eq:psi-tern-L-large}
		\left| R(\ell; g)^0 \cap \Psi(\delta, 2L) \right| \\ \prec \ell^{\frac{1}{2}} L \left( 1+ \ell^{\frac{1}{2}}H(g) \delta^{\frac{1}{2}} + \frac{\ell}{N^{\frac{1}{2}}}L + \ell H(g) \delta^{\frac{1}{2}} L+\frac{\ell^{2}}{N} \delta^{\frac{1}{2}} L^2 \right)
	\end{equation}
	and by \cite[Theorem 2.6]{SupThetaII} that
	\begin{equation} \label{eq:psi-typeII-L-large}
		\max_{\frac{1}{\ell}\mathbb{Z}} \left| R(\ell; g)^0 \cap \Omega(1, 2L) \cap \det{}^{-1}(\{n\})  \right| \prec \ell^{\frac{1}{2}} \delta^{-\frac{1}{2}} \left( 1+ \frac{\ell^2}{N} \delta^{\frac{1}{2}} L^2  \right).
	\end{equation}
	The split case of Proposition \ref{prop:Psi-lattice-counting} now follows from combining \eqref{eq:psi-reduct-to-ternary}, \eqref{eq:psi-tern-L-large}, and \eqref{eq:psi-typeII-L-large}.
\end{proof}

\subsection{A geometric view on the conjecture}
To give the reader some intuition for this new condition, we give some geometric interpretation of it. For $z,w \in \mathbb{H}$, let
$$
u(z,w) = \frac{|z-w|^2}{4 \Im z \Im w},
$$
which relates to the hyperbolic distance through $d(z,w) = \arccosh(1+2u(z,w))$. Then, for small $\delta$
$$\begin{aligned}
\begin{Bmatrix}
	\gamma_1, \gamma_2 \in R(1;g) \\
	\gamma_1, \gamma_2 \in \Omega(\delta/4,2L) \backslash \Omega(\delta,L)  \\
	 \det(\gamma_1) = \det(\gamma_2) \\
	|\diamondsuit(\gamma_1)-\diamondsuit(\gamma_2)| \le \heartsuit (2L)^4
\end{Bmatrix} 
&\overset{z=gi}{\leftrightsquigarrow}
\begin{Bmatrix}
	\gamma_1,\gamma_2 \in R \\
	u(z, \gamma_1 z), u(z, \gamma_2 z) \asymp \delta \\
	\det(\gamma_1) = \det(\gamma_2) \asymp L \\
	|u(z, \gamma_1 z) - u(z, \gamma_2 z)| \ll \heartsuit
\end{Bmatrix}, \\
\begin{Bmatrix}
	\gamma_1, \gamma_2 \in R(1;g) \\
	\gamma_1, \gamma_2 \in \Psi(\delta/4,2L) \backslash \Psi(\delta,L)  \\
	\det(\gamma_1) = \det(\gamma_2) \\
	|\diamondsuit(\gamma_1)-\diamondsuit(\gamma_2)| \le \heartsuit (2L)^4
\end{Bmatrix} 
&\overset{z=gi}{\leftrightsquigarrow}
\begin{Bmatrix}
	\gamma_1,\gamma_2 \in R \\
	u(-\overline{z}, \gamma_1 z), u(-\overline{z}, \gamma_2 z) \asymp \delta \\
	\det(\gamma_1) =  \det(\gamma_2) \asymp L \\
	|u(-\overline{z}, \gamma_1 z) - u(-\overline{z}, \gamma_2 z)| \ll \heartsuit
\end{Bmatrix}.
\end{aligned}$$
These should be understood as imprecise correspondences that could be made more precise through additional technicalities. When $\gamma_1, \gamma_2 \in \Omega(\delta/4,2l) \backslash \Omega(\delta,L)$, we see that the new condition corresponds to the condition that the two matrices $\gamma_1, \gamma_2$ move the point $z=gi$ approximately the same amount. When the sets $\Psi$ are considered instead, the condition is is similar, though the distance is measured from $-\overline{z}$ which has to do with the fact that this concerns negative determinants which one should think of as involving the reflection operator.

\subsection{Further motivation of the conjecture}
We provide the following evidence towards this conjecture. First, we give a heuristic for the volume term, i.e.\@ the largest scale, and second, we provide an upper bound for the upper-triangular matrices when $B$ is the matrix algebra.

\subsubsection{Volume heuristic}
For the heuristic of the volume term, we have:
\begin{itemize}
	\item the areas of $\Omega(\delta,L)$ and $\Psi(\delta,L)$ are both $\asymp \delta L^4$;
	\item the co-volume of $R(\ell;g)$ is $\asymp d_B N / \ell^2$;
	\item the relative size of the equation $\det(\gamma_1)=\det(\gamma_2)$ is $L^2 / \ell$ as $|\det(\gamma_i)| \le L^2$ and $\det(\gamma_i) \in \frac{1}{\ell} \mathbb{Z}$ for $i=1,2$; and
	\item the density of the inequality $|\diamondsuit(\gamma_1)-\diamondsuit(\gamma_2)| \le \heartsuit L^4$ is $\asymp \min\{\heartsuit, \delta\} / \delta$ as $\diamondsuit(\gamma_1)$ and $\diamondsuit(\gamma_2)$ are typically of size $\delta L^4$.
\end{itemize}
Hence,
\begin{equation}
	\left(\delta L^4\right)^2 \times \left(\frac{d_BN}{\ell^2}\right)^{-2} \times \left(\frac{L^2}{\ell}\right)^{-1} \times \frac{\min\{\heartsuit, \delta \}}{\delta} = \frac{\ell^5}{(d_BN)^2} \min\{\heartsuit,\delta\} \delta L^6,
\end{equation}
which is smaller than the volume term of Conjecture \ref{conj:heart-lattice-counting} by a factor of $\delta$.

\subsubsection{Upper-triangular matrices} In the case, where $B$ is the matrix algebra and $R = \sm \mathbb{Z} & \mathbb{Z} \\ N \mathbb{Z} & \mathbb{Z} \esm$, we shall upper-bound the contribution of the upper-triangular matrices to Conjecture \ref{conj:heart-lattice-counting}.



Since $R(\ell)$ is invariant under the Atkin--Lehner operators $A_0(N)$, we may assume that $g i = z = x+iy$ with $y$ maximal under the orbit of $A_0(N)$ such that $H(g)=y$. Since $\Omega^\star(\delta,L)$, $\Psi^\star(\delta,L)$, $\det(\cdot)$, and $\diamondsuit(\cdot)$ are all bi-$K_{\infty}$-invariant, we may also suppose that $g$ is upper-triangular, i.e.\@ $g= \sm 1 & x \\  & 1 \esm \sm y^{1/2} & \\ & y^{-1/2} \esm $.
	Let
	$$
	\gamma_i = [b_i,b_i,c_i] + d_i = \begin{pmatrix} d_i+c_i & 2b_i \\  & d_i-c_i	\end{pmatrix} \in R(\ell), \quad i=1,2,
	$$
	be upper-triangular. Thus, $c_i, d_i \in \frac{1}{2}\mathbb{Z}$ and $b_i \in \frac{1}{2\ell} \mathbb{Z}$ for $i=1,2$. By assumption,
	\begin{equation} \label{eq:counting-domain}
		g^{-1} \gamma_i g = \begin{pmatrix} d_i+c_i & \frac{2(c_ix+b_i)}{y} \\ & d_i-c_i  \end{pmatrix} \in \Omega^\star(\delta,L) \cup \Psi^\star(\delta,L), \quad i=1,2,
	\end{equation}
	\begin{equation} \label{eq:counting-det}
		d_1^2-c_1^2 = d_2^2-c_2^2,
	\end{equation}
	\begin{multline} \label{eq:counting-diamond}
		\Biggl |  \left( \left( \frac{c_1x+b_1}{y}  \right)^2  +  d_1^2 \right)\left(\left( \frac{c_1x+b_1}{y} \right)^2 + c_1^2 \right) \\ -
		\left( \left( \frac{c_2x+b_2}{y}  \right)^2  +  d_2^2 \right)\left(\left( \frac{c_2x+b_2}{y} \right)^2 + c_2^2 \right)    
		    \Biggr| \le \tfrac{1}{4} \heartsuit L^4, \quad i=1,2.
	\end{multline}
	Though, we won't make use of the last inequality. An application of Cauchy--Schwarz allows us to assume either both $\gamma_1,\gamma_2 \in g\Omega^\star(\delta,L)g^{-1}$ or $g\Psi^\star(\delta,L)g^{-1}$. We treat these cases separately.
	
	\paragraph{Case: $\gamma_1,  \gamma_2 \in g\Omega^\star(\delta,L)g^{-1}$} Here, \eqref{eq:counting-domain} implies:
	$$
	\left| c_ix+b_i  \right| \le y\delta^{\frac{1}{2}}L, \quad |c_i| \le \delta^{\frac{1}{2}}L, \quad |d_i| \le L, \quad i=1,2.
	$$
	Let $L \ll \min\{\ell^{-\frac{1}{2}}, \ell^{-1} y^{-1} \delta^{-\frac{1}{2}}\}$. This forces $c_1=d_1=0$ which, in turn, also forces $b_1=0$. Hence $\gamma_1 = 0$, which is excluded from the count. Thus, we shall assume $L \gg \min\{\ell^{-\frac{1}{2}}, \ell^{-1} y^{-1} \delta^{-\frac{1}{2}} \}$ as otherwise the set to be counted is empty. We must differentiate whether $c_1 = \pm c_2$ or not. If $c_1 = \pm c_2$, then we have at most $O(\delta^{\frac{1}{2}}L)$ choices for $c_1$ and at most two for $c_2$. Furter, we have at most $O(1+ \ell y\delta^{\frac{1}{2}}L)$ choices for $b_1, b_2$, respectively, and $O(1+L)$ choices for $d_1$, whereas $d_2 = \pm d_1$ by \eqref{eq:counting-det}. This leads to a contribution of
	$$
	\ll (1+\delta^{\frac{1}{2}}L) (1+L) \left(1+\ell y \delta^{\frac{1}{2}} L\right)^2 \ll \ell L^2\left(1+\ell H(g)^2 \delta + \ell H(g)^2 \delta L+\ell H(g)^2 \delta^{\frac{3}{2}}L^2\right).
	$$
	If $c_1 \neq \pm c_2$, then there are $O(1+\delta^{\frac{1}{2}}L)$ choices for $c_1, c_2$, respectively. By \eqref{eq:counting-det} and the divisor bound, there are at most $\prec 1$ choices for $d_1, d_2$. Again, there are at most $O(1+\ell y \delta^{\frac{1}{2}}L)$ choices for $b_1, b_2$, respectively. This leads to a contribution of
	$$
	\prec (1+\delta^{\frac{1}{2}}L)^2 \left(1+\ell y \delta^{\frac{1}{2}} L\right)^2 \prec \ell L^2\left(1+\ell H(g)^2 \delta + \ell H(g)^2 \delta L+\ell H(g)^2 \delta^{\frac{3}{2}}L^2\right),
	$$
	which is smaller than what is claimed in Conjecture \ref{conj:heart-lattice-counting}.
	
	\paragraph{Case: $\gamma_1, \gamma_2 \in g \Psi^\star(\delta,L)g^{-1}$} Here, \eqref{eq:counting-domain} implies:
	$$
	\left| c_ix+b_i  \right| \le y\delta^{\frac{1}{2}}L, \quad |c_i| \le L, \quad |d_i| \le \delta^{\frac{1}{2}} L, \quad i=1,2.
	$$
	We see that the same argument as in the previous case may be applied upon reversing the role of $c_i$ and $d_i$ for $i=1,2$.

%



\section{Archimedean test functions}
\label{sec:arch-test-func}

\subsection{The subspace $V_{m,\omega}$ and $\Omega_{\infty}$}
\label{sec:weil-rep}

Let $\rho$ denote the Weil representation of $\SL{2}(\mathbb{R})$ on $L^2(B_{\infty})$ defined with respect to the additive character $\psi_{\infty}(\xi)= e(\xi)$. In \cite[\S 3]{SupThetaI}, a certain invariant subspace $\Omega_{\infty} \subset L^2(B_{\infty})$ was considered. We shall briefly recollect some necessary facts which may be found in the mentioned reference.

$L^2(B_{\infty})$ splits into the completed orthogonal direct sum $\bigobot_{m \in \mathbb{Z}} V_{m,\omega}$, where $V_{m,\omega}$ is the $(\rho(D_{\omega}^{-1}K_{\infty}D_{\omega},e^{im\varphi})$-isotypical part, where $D_{\omega}= {\rm diag}(\sqrt{2\pi/\omega}, \sqrt{\omega/2\pi})$, i.e.\@ for $\Phi \in V_{m,\omega}$ and $k_{\varphi} = \sm \cos(\varphi) & \sin(\varphi) \\ -\sin(\varphi) & \cos( \varphi) \esm \in K_{\infty} $, we have $\rho(D_{\omega}^{-1}k_{\varphi}D_{\omega})\Phi( \gamma) = e^{im\varphi} \Phi(\gamma)$. A dense subset of $V_{m,\omega}$ is given by the Schwartz functions $\Phi$ satisfying the quantum harmonic oscillator partial differential equation
\begin{equation}
	-\Delta \Phi(\gamma)+\omega^2 \det(\gamma) \Phi(\gamma) = \omega m \Phi(\gamma),
	\label{eq:isotypical-PDE}
\end{equation}
where the Laplace operator is defined as having Fourier multiplier $-4\pi^2 \det(\cdot)$. The subspace $\Omega_{\infty}$ is now defined as the finite linear span
\begin{multline*}
{\rm Span}_{\mathbb{C}} \Bigl\{ \psi_{\infty}(\sigma \det(\gamma))\Phi(\gamma) \Big \vert \sigma \in \mathbb{R}, \exists \omega >0: \Phi \in V_{0,\omega}\\ \text{ and } \exists \delta >0: |\Phi(\gamma)| \ll (1+P(\gamma))^{-2-\delta}    \Bigr\}.
\end{multline*}

\subsection{A family of test functions}

In this section, we construct bi-$K_{\infty}$-invariant archi\-me\-dean test functions $\Phi_{\infty}$ on $B_{\infty}$ in the subspace $\Omega_{\infty} \cap V_{0,2\pi} $ with a prescribed Selberg/Harish-Chandra transform of a fairly general class. We shall now describe what we precisely mean by the latter.

Denote with $\Xi_s$ the spherical function for the principal series on $\PGL{2}(\mathbb{R})$. For $g \in \PGL{2}(\mathbb{R})$, the spherical function may be explicitly expressed as
\begin{equation}\begin{aligned}
\Xi_s(g) 
&= P_{-s}(2u(g)+1) \\
&= {}_2F_1(s,1-s;1;-u(g)),
\label{eq:matrixcoef}\end{aligned}
\end{equation}
where $P_{-s}(x)$ denotes the associated Legendre function and for $\gamma \in \GL{2}(\mathbb{R})$
\begin{equation} \label{eq:u-def-first}
u(\gamma) = \frac{1}{2} \frac{P(\gamma)-|\tau(\gamma)|}{|\tau(\gamma)|},
\end{equation}
where we recall the definition of $P(\gamma)$ and $\tau(\gamma)$ from \eqref{eq:P-tau-diam-def}. We shall abuse some notation and also write $\Xi_s(u)=P_{-s}(2u+1)$, so that $\Xi_s(g)=\Xi_s(u(g))$. By the nature of the hypergeometric function ${}_2F_1$, the function $\Xi_s(u)$ satisfies the differential equation
\begin{equation}
	u(u+1) \Xi_s''(u)+(2u+1)\Xi_s'(u)+s(1-s)\Xi_s(u) = 0.
\label{eq:Xi-diff-eq}
\end{equation}

A bi-$K_{\infty}$-invariant function $\Phi_{\infty}$ may be expressed solely in terms of $P$ and $\tau$ or $u$ and $\tau$, though in the latter one must be careful about the locus $\tau=0$. For $\tau(\gamma) \neq 0$, we shall write
$$
\Phi_{\infty}(\gamma) = \frac{1}{|\tau(\gamma)|} k(u(\gamma); \tau(\gamma)).
$$
The Selberg/Harish-Chandra of $k$ (respectively $\Phi_{\infty}$) on the locus $\tau \neq 0$ is given by
\begin{equation}
	h(t;\tau) = 4 \pi \int_0^{\infty} k(u;\tau) \Xi_{\frac{1}{2}+it}(u) du.
	\label{eq:selberg}
\end{equation}
The remainder of this section is devoted to proving the following theorem.

\begin{thm} \label{thm:Existence-test-function} Suppose $h(t;\tau)$ is of the shape
	\begin{equation}
		h(t;\tau) = h(t) \cosh(\alpha t) \cdot 2\sqrt{|\tau|}K_{it}(2 \pi |\tau|),
		\label{eq:h-shape}
	\end{equation}
	for $\tau \neq 0$ and some $0 \le \alpha < \frac{\pi}{2}$, where $h$ is either 
	\begin{enumerate}
		\item a (sum of) function(s) of the type $\cos( \beta t)$ with $\beta \in \mathbb{R}$, or \label{enum:thm3-type1}
		\item satisfies both \label{enum:thm3-type2}
	\begin{enumerate}
		\item $h(t)$ is entire and even, and \label{enum:thm3-a}
		\item for any $A> 0$: $|h(t)| \ll_{A} (1+|t|)^{-1-\delta}$ for some $\delta >0$ in the strip $|\Im(t)| \le A$. \label{enum:thm3-b}
	\end{enumerate}
\end{enumerate}
	Then, the function
	\begin{equation}
		\Phi_{\infty}(\gamma) = \frac{1}{|\tau(\gamma)|} k(u(\gamma); \tau(\gamma)),
		\label{eq:M-k-relation}
	\end{equation}
where $k(u;\tau)$ is given by
	\begin{equation}
	k(u;\tau) = \frac{1}{4 \pi } \int_{-\infty}^{\infty} h(t;\tau) \Xi_{\frac{1}{2}+it}(u) \tanh(\pi t) t dt,
		\label{eq:inverse-selberg}
	\end{equation}
	for $\tau(\gamma) \neq 0$, extends to a continuous bi-$K_{\infty}$-invariant function $\Phi_{\infty}(\gamma) \in V_{0,2\pi} \subset \Omega_{\infty}$ with Selberg/Harish-Chandra transform $h(t;\tau)$.
\end{thm}

\begin{rem}
	It seems likely that, with some additional care, Theorem \ref{thm:Existence-test-function} may be extended to include the case $\alpha= \frac{\pi}{2}$, though further restrictions on the function $h(t)$ may be necessary.
\end{rem}

Since $h( \cdot ; \tau)$ satisfies the classical inversion criteria (see for example \cite[Pages 31,32]{IwSpecMeth}), the formula \eqref{eq:inverse-selberg} holds pointwise for $\tau$. Likewise, the bi-$K_{\infty}$-invariance of $\Phi_{\infty}$ is immediately clear. The main difficulty is showing that $\Phi_{\infty} \in V_{0,2\pi} \subset \Omega_{\infty}$.


\subsection{A Selberg/Harish-Chandra/Mehler--Fock inversion problem}

Let $h(t;\tau)$,\\ $k(u;\tau)$, and $\Phi_{\infty}(\gamma)$ be as in Theorem \ref{thm:Existence-test-function}. We assume the parameter $\alpha$ to be fixed and thus we allow any implied constant to depend on $\alpha$ without further mention. The goal of this section is to establish the continuity of $\Phi_{\infty}$ as well as certain decay properties.

The inverse Selberg/Harish-Chandra transform may be computed in three steps: an inverse Fourier transform, a change of variables, and an inverse Abel transform. Explicitly, we have
$$\begin{aligned}
	g(r;\tau) &= \frac{1}{2\pi} \int_{-\infty}^{\infty} h(t;\tau) e^{irt} dt,\\
	q(v;\tau) &= \tfrac{1}{2} \, g(\arccosh(2v+1);\tau), \\
	k(u;\tau) &= -\frac{1}{\pi} \int_0^{\infty} \frac{1}{\sqrt{v}} dq(v+u;\tau).
\end{aligned}$$
The inverse Fourier transform of $\cosh(\alpha t) \cdot 2 \sqrt{|\tau|} K_{it}(2\pi |\tau|)$ is given by (see \cite[Eqs.\@ 2.16.48.9 \& 2.16.48.20]{Prudnikov})
\begin{equation}
	\frac{\sqrt{|\tau|}}{\pi} \int_{- \infty}^{\infty} \cosh(\alpha t)  K_{it}(2\pi |\tau|) e^{i r t} dt \\
	= \sqrt{|\tau|} e^{-2 \pi |\tau| \cos(\alpha) \cosh( r)} \cos\left( 2 \pi |\tau| \sin(\alpha) \sinh( r) \right).
	\label{eq:K-Bessel-FT}
\end{equation}
We assume from now on that $h$ is of the second type \emph{(\ref{enum:thm3-type2})} in Theorem \ref{thm:Existence-test-function}. We omit treating the case where $h$ is of the first type \emph{(\ref{enum:thm3-type1})}, as in this case $h$ is the Fourier transform of a finite sum of point masses and we trust the reader to be able to make suitable adjustments of the subsequent arguments.

Let $g(r)=\frac{1}{2\pi} \int_{-\infty}^{\infty} h(t) e^{irt}dt$ denote the inverse Fourier transform of $h(t)$. It is an even function and satisfies 
\begin{equation}
	g(r) = \frac{1}{2\pi} \int_{-\infty}^{\infty} 
	h(t) e^{irt} dt = \frac{1}{2\pi} \int_{-\infty}^{\infty} 
	h(t+\sign(r)iA) e^{irt}e^{-\sign(r)Ar} dt \ll_{A} e^{-A|r|}, \label{eq:g-decay}
\end{equation}
where we have used a contour shift.
We may now express $g(r;\tau)$ as the convolution 
\begin{equation}
	g(r;\tau) = \sqrt{|\tau|} \int_{-\infty}^{\infty} g(\ell)  e^{-2 \pi |\tau| \cos(\alpha) \cosh(r- \ell)} \cos\left( 2 \pi |\tau| \sin(\alpha) \sinh(r- \ell) \right)  d\ell.
	\label{eq:g-r-tau-convolution}
\end{equation}
Hence,
\begin{multline}
	q(v;\tau) = \frac{\sqrt{|\tau|}}{2} \int_{-\infty}^{\infty} g(\ell)  e^{-2 \pi |\tau| \cos(\alpha) [(2v+1)\cosh(\ell)-2\sqrt{v(v+1)}\sinh(\ell)]} \\ \cos\left( 2 \pi |\tau| \sin(\alpha) \left[2\sqrt{v(v+1)}\cosh(\ell)-(2v+1)\sinh(\ell)\right] \right)  d\ell.
	\label{eq:q-v-tau-convolution}
\end{multline}
Let
\begin{multline}
	Q(P;\tau) = \frac{1}{2} \int_{-\infty}^{\infty} g(\ell)  e^{-2 \pi \cos(\alpha) [P\cosh(\ell)-\sqrt{P^2-\tau^2}\sinh(\ell)]} \\ \cos\left( 2 \pi  \sin(\alpha) \left[\sqrt{P^2-\tau^2}\cosh(\ell)-P\sinh(\ell)\right] \right)  d\ell,
	\label{eq:Q-P-tau-convolution}
\end{multline}
for $P \ge |\tau|$, such that
$$
q(v;\tau) = \sqrt{|\tau|} Q(|\tau|(2v+1);\tau).
$$
Then, we may write
$$\begin{aligned}
	k(u,\tau) &= - \frac{1}{\pi} \int_0^{\infty} \frac{1}{\sqrt{v}} dq(v+u;\tau) \\ &= -\frac{\sqrt{2}|\tau|}{\pi} \int_0^{\infty} \frac{1}{\sqrt{2v|\tau|}} dQ(|\tau| (2v+2u+1);\tau) \\
	&= - \frac{\sqrt{2}|\tau|}{\pi} \int_0^{\infty} \frac{1}{\sqrt{v}} dQ(v+(2u+1)|\tau|; \tau),
\end{aligned}$$
where integration takes place in the variable $v$. Hence,
\begin{equation}
	\Phi_{\infty}(\gamma) 
	= -  \frac{\sqrt{2}}{\pi} \int_0^{\infty} \frac{1}{\sqrt{v}} dQ(v+P(\gamma);\tau(\gamma)).
	\label{eq:M-def-2}
\end{equation}

\begin{lem} \label{lem:Q-upper-bounds} We have for any $A \ge 0$
	$$\begin{aligned}
		Q(P;\tau) &\ll_{A} (1+P)^{-A}, \\
		Q_P(P;\tau) &\ll_A (1+P)^{-A}, \\
		Q_{PP}(P;\tau) &\ll_A \left(1+\frac{P^2}{P^2-\tau^2} \right) (1+P)^{-A}, \\
		Q_{P\tau}(P;\tau) &\ll_A \left(1+\frac{P^2}{P^2-\tau^2} \right) (1+P)^{-A} .
	\end{aligned}$$
\end{lem}

\begin{proof} We begin by writing the integral defining $Q$ symmetric around the origin in $\ell$: 

\begin{equation}\begin{aligned}
	Q(P;\tau) =& \frac{1}{2} \int_{-\infty}^{\infty}  g(\ell) e^{-2\pi \cos(\alpha) P \cosh(\ell)} \\ & \times \Biggl[  \cosh\left(2\pi \cos(\alpha) \sqrt{P^2-\tau^2} \sinh(\ell)\right)  \cos\left(2\pi \sin(\alpha) \sqrt{P^2-\tau^2} \cosh(\ell)\right) \\ & \hspace{8.5cm} \times \cos\left(2\pi \sin(\alpha) P \sinh(\ell)\right) \\
	  & \quad +  \sinh\left(2\pi \cos(\alpha) \sqrt{P^2-\tau^2} \sinh(\ell)\right) \sin\left(2\pi \sin(\alpha) \sqrt{P^2-\tau^2} \cosh(\ell)\right) \\ & \hspace{8.5cm} \times \sin\left(2\pi \sin(\alpha) P \sinh(\ell)\right) \Biggr]  
	 d\ell.
	\label{eq:Q-symmetric}
\end{aligned}
\end{equation}
If we let $\cosh(r) = P/|\tau|$ with $r \ge 0$, then we have
$$\begin{aligned}
Q(P;\tau) &\ll \int_{-\infty}^{\infty} |g(\ell)| e^{-2\pi \cos(\alpha) P \cosh(\ell)} \cosh\left(2\pi \cos(\alpha) \sqrt{P^2-\tau^2} \sinh(\ell)\right) d\ell \\
&\ll \sum_{\pm} \int_{-\infty}^{\infty} |g(\ell)| e^{-2\pi \cos(\alpha) |\tau| \cosh(r \pm \ell)} d\ell \\
& \ll_{A} \sum_{\pm} \int_{-\infty}^{\infty} e^{-(A+1)|\ell|} \left(1+|\tau|e^{|r\pm \ell|}\right)^{-A} d\ell \\
&\ll_{A} \sum_{\pm} \int_{-\infty}^{\infty} e^{-|\ell|} \left(1+|\tau| e^{|r|} \right)^{-A} d\ell \\
& \ll_{A} (1+P)^{-A},
\end{aligned}$$
where we have used the super-exponential decay of $g$, see \eqref{eq:g-decay}. Similarly, we have
$$\begin{aligned}
	Q_P(P;\tau) &\ll \int_{-\infty}^{\infty} |g(\ell)| \left(1+\frac{P}{\sqrt{P^2-\tau^2}} \min\left\{1, \sqrt{P^2-\tau^2} \cosh(\ell) \right\}\right) \cosh(\ell)  \\
	&\quad \times e^{-2\pi \cos(\alpha) P \cosh(\ell)} \cosh\left(2\pi \cos(\alpha) \sqrt{P^2-\tau^2} \sinh(\ell)\right) d\ell \\
	&\ll \left(1+P\right) \sum_{\pm} \int_{-\infty}^{\infty} |g(\ell)| \cosh(\ell)^2  e^{-2\pi \cos(\alpha) \cosh(r \pm \ell)} d\ell \\
	& \ll_A  (1+P)^{-A}
\end{aligned}$$
and
$$\begin{aligned}
	Q_{PP}(P;\tau) &\ll \int_{-\infty}^{\infty}  |g(\ell)| \cosh(\ell)^2 e^{-2\pi \cos(\alpha) P \cosh(\ell)} \cosh\left(2\pi \cos(\alpha) \sqrt{P^2-\tau^2} \sinh(\ell)\right)  \\
	&\quad \times \!\!  \left[1+\left(\frac{1+P}{\sqrt{P^2-\tau^2}}+ \frac{P^2}{(P^2-\tau^2)^{\frac{3}{2}}} \right) \min\left\{1, \sqrt{P^2-\tau^2} \cosh(\ell) \right\}+\frac{P^2}{P^2-\tau^2} \right] \! d\ell \\
	&\ll \left(1+P+\frac{P^2}{P^2-\tau^2}\right) \sum_{\pm} \int_{-\infty}^{\infty} |g(\ell)| \cosh(\ell)^3  e^{-2\pi \cos(\alpha) \cosh(r \pm \ell)} d\ell \\
	& \ll_A  \left(1+\frac{P^2}{P^2-\tau^2} \right) (1+P)^{-A}.
\end{aligned}$$
The last estimate equally holds for $Q_{P\tau}(P;\tau)$.
\end{proof}

\begin{cor} 	\label{cor:M-infty-bounds} $\Phi_{\infty}$ is continuous and smooth as long as $P(\gamma) > |\tau(\gamma)|$. Furthermore, we have the following bounds,
	$$\begin{aligned}
		\Phi_{\infty}(\gamma) &\ll_A (1+P(\gamma))^{-A},\\
		D\Phi_{\infty}(\gamma) & \ll_{D,A,\epsilon}   \left(P(\gamma)^{\frac{1}{2}}+ \frac{P(\gamma)^2 }{(P(\gamma)^2-\tau(\gamma)^2)^{\frac{1}{2}+\epsilon}} \right) (1+P(\gamma))^{-A},
	\end{aligned}$$
 where $D$ is any differential of order $1$. In particular, we have $\Phi_{\infty} \in L^2(B_{\infty})$.
\end{cor}
\begin{proof} We have
	$$\begin{aligned}
	\Phi_{\infty}(\gamma) &\ll \int_0^{\infty} \frac{1}{\sqrt{v}} Q_P(v+P(\gamma);\tau(\gamma)) dv \\
	&\ll_A \int_0^{\infty} \frac{1}{\sqrt{v}} \left(1+ v+P(\gamma) \right)^{-A-1} dv \\
	&\ll_A (1+P(\gamma))^{-A}.
	\end{aligned}$$
	Likewise, we have
	$$\begin{aligned}
	D\Phi_{\infty}(\gamma) &\ll_D \int_0^{\infty} \frac{1}{\sqrt{v}} P(\gamma)^{\frac{1}{2}} \left(|Q_{PP}(v+P(\gamma);\tau(\gamma))| + |Q_{P\tau}(v+P(\gamma);\tau(\gamma))| \right) dv \\
	&\ll_{D,A} P(\gamma)^{\frac{1}{2}} \int_0^{\infty} \frac{1}{\sqrt{v}} \left(1+\frac{v^2+P(\gamma)^2}{v^2+vP(\gamma)+P(\gamma)^2-\tau(\gamma)^2}\right) \left(1+ v+P(\gamma) \right)^{-A-2} dv \\
	&\ll_{D,A,\epsilon}   \left(1+ \frac{P(\gamma)^2 \cdot  (P(\gamma)^2-\tau(\gamma)^2)^{-\epsilon}}{(P(\gamma)^2-\tau(\gamma)^2)+(P(\gamma)^2-\tau(\gamma)^2)^{\frac{3}{4}}+P(\gamma)^{\frac{1}{2}}(P(\gamma)^2-\tau(\gamma)^2)^{\frac{1}{2}}} \right) \\
	& \quad \times P(\gamma)^{\frac{1}{2}} (1+P(\gamma))^{-A}.
	\end{aligned}$$
	The continuity as well as the smoothness for $P(\gamma) > |\tau(\gamma)|$ are also clear as differntiation only yields denominators being a power of $P^2-\tau^2$ and the extra powers of $\cosh(\ell)$ are compensated by the super-exponentially decaying function $g$, see \eqref{eq:g-decay}.

\end{proof}

\subsection{Quantum harmonic oscillator}

Let $h(t;\tau)$ and $\Phi_{\infty}$ be as in Theorem \ref{thm:Existence-test-function}. We shall assume again, that $\alpha$ is fixed and any implied constants may depend on $\alpha$ without further mention. $h(t;\tau)$ satisfies the differential equation
\begin{equation}
	h_{\tau\tau}(t;\tau) + \left( -4 \pi^2+\frac{\frac{1}{4}+t^2}{\tau^2} \right)h(t;\tau) = 0.
	\label{eq:h-diff-eq}
\end{equation}
This follows immediately from the differential equation for the $K$-Bessel function.

\begin{lem} \label{lem:PDE} Whenever $P(\gamma)>|\tau(\gamma)|$, $\Phi_{\infty}$ satisfies the quantum harmonic oscillator partial differential equation
	\begin{equation}
		-\Delta \Phi_{\infty} (\gamma) + 4 \pi^2 \det(\gamma) \Phi_{\infty}(\gamma)= 0,
		\label{eq:spherical-PDE}
	\end{equation}
	where the Laplace operator $\Delta$ is defined as having Fourier multiplier $-4\pi^2 \det(\gamma)$. Moreover, we have $\Phi_{\infty} \in V_{0,2\pi} \subset \Omega_{\infty}$.
\end{lem}
\begin{proof} We shall temporarily drop the subscript $\infty$ from $\Phi_{\infty}$ as to leave space for subscripts indicating partial derivatives. Let $\gamma = [a,b,c]+d \in B_{\infty}$, then $\Delta$ is given explicitly by $\frac{1}{4} (\frac{\partial^2}{\partial a^2}-\frac{\partial^2}{\partial b^2}-\frac{\partial^2}{\partial c^2}+\frac{\partial^2}{\partial d^2})$. Since $\Phi$ only depends on $\tau$ and $P$, we find
	$$
	\Delta \Phi =  \tau \Phi_{PP} + 2 P \Phi_{P\tau} +2 \Phi_{\tau} +\tau \Phi_{\tau \tau}.
	$$
	Assume now $P > |\tau|$ and suppose first $\tau > 0$. Due to Lebesgue dominated convergence (see Appendix \ref{app:bounds-special} Lemmata \ref{lem:h(t,tau)-bounds}, \ref{lem:Xi-derivative-bounds} for bounds on the special functions), we may compute: 
	{\allowdisplaybreaks
		\begin{align}
			\Phi(\gamma) &=	\frac{1}{4 \pi} \int_{-\infty}^{\infty} \frac{1}{\tau} \Xi_{\frac{1}{2}+it}\left( \frac{P}{2 \tau}- \frac{1}{2} \right) h(t;\tau) \tanh(\pi t) t dt, \\
			\! \Phi_P(\gamma) &=	\frac{1}{4 \pi} \int_{-\infty}^{\infty} \frac{1}{2\tau^2} \Xi'_{\frac{1}{2}+it}\left( \frac{P}{2 \tau}- \frac{1}{2} \right) h(t;\tau) \tanh(\pi t) t dt, \\
			\Phi_{\tau}(\gamma) &= \begin{multlined}[t]	\frac{1}{4 \pi} \int_{-\infty}^{\infty} \Biggl[- \frac{1}{\tau^2} \Xi_{\frac{1}{2}+it}\left( \frac{P}{2 \tau}- \frac{1}{2} \right) h(t;\tau)-\frac{P}{2\tau^3} \Xi'_{\frac{1}{2}+it}\left( \frac{P}{2 \tau}-\frac{1}{2}\right) h(t;\tau) \\
				+\frac{1}{\tau} \Xi_{\frac{1}{2}+it}\left(\frac{P}{2 \tau}-\frac{1}{2} \right) h_{\tau}(t;\tau) \Biggr	] \tanh(\pi t) t dt,
			\end{multlined} \\
			\Phi_{PP}(\gamma) &= \frac{1}{4 \pi} \int_{-\infty}^{\infty} \frac{1}{4\tau^3} \Xi''_{\frac{1}{2}+it}\left( \frac{P}{2 \tau}- \frac{1}{2} \right) h(t;\tau)  \tanh(\pi t) t dt,\\
			\Phi_{P\tau}(\gamma) &= \begin{multlined}[t] \frac{1}{4 \pi} \int_{-\infty}^{\infty} \Biggl(- \frac{1}{\tau^3} \Xi'_{\frac{1}{2}+it}\left( \frac{P}{2 \tau}- \frac{1}{2} \right) h(t;\tau)-\frac{P}{4\tau^4} \Xi''_{\frac{1}{2}+it}\left( \frac{P}{2 \tau}-\frac{1}{2}\right) h(t;\tau) \\
				+\frac{1}{2\tau^2} \Xi'_{\frac{1}{2}+it}\left(\frac{P}{2 \tau}-\frac{1}{2} \right) h_{\tau}(t;\tau) \Biggr) \tanh(\pi t) t dt, \end{multlined}\\
			\Phi_{\tau \tau}(\gamma) &= \begin{multlined}[t] \frac{1}{4 \pi} \int_{-\infty}^{\infty} \Biggl( \frac{2}{\tau^3} \Xi_{\frac{1}{2}+it}\left( \frac{P}{2 \tau}- \frac{1}{2} \right) h(t;\tau)+\frac{P}{2\tau^4} \Xi'_{\frac{1}{2}+it}\left( \frac{P}{2 \tau}-\frac{1}{2}\right) h(t;\tau)\\
				-\frac{1}{\tau^2} \Xi_{\frac{1}{2}+it}\left(\frac{P}{2 \tau}-\frac{1}{2} \right) h_{\tau}(t;\tau) + \frac{3P}{2\tau^4} \Xi'_{\frac{1}{2}+it}\left( \frac{P}{2 \tau}- \frac{1}{2} \right) h(t;\tau) \\
				+\frac{P^2}{4\tau^5} \Xi''_{\frac{1}{2}+it}\left( \frac{P}{2 \tau}- \frac{1}{2} \right) h(t;\tau) - \frac{P}{2 \tau^3} \Xi'_{\frac{1}{2}+it}\left( \frac{P}{2 \tau}- \frac{1}{2} \right) h_{\tau}(t;\tau) \\
				- \frac{1}{\tau^2} \Xi_{\frac{1}{2}+it}\left( \frac{P}{2 \tau}- \frac{1}{2} \right) h_{\tau}(t;\tau) - \frac{P}{2 \tau^3} \Xi'_{\frac{1}{2}+it}\left( \frac{P}{2 \tau}- \frac{1}{2} \right) h_{\tau}(t;\tau) \\
				+ \frac{1}{\tau} \Xi_{\frac{1}{2}+it}\left( \frac{P}{2 \tau}- \frac{1}{2} \right) h_{\tau \tau}(t;\tau) \Biggr) \tanh(\pi t) t dt.\end{multlined}
		\end{align}
	}
	We conclude
	\begin{align*}
		\Delta \Phi(\gamma) &=	\frac{1}{4 \pi} \int_{-\infty}^{\infty} \Biggl( \Xi''_{\frac{1}{2}+it}\left( \frac{P}{2 \tau}- \frac{1}{2} \right) \frac{1}{\tau^2} \left(\frac{1}{4}- \frac{P^2}{4\tau^2}\right) \\
		&\quad - \Xi'_{\frac{1}{2}+it}\left( \frac{P}{2 \tau}- \frac{1}{2} \right) \frac{P}{\tau^3}  h(t;\tau) + \Xi_{\frac{1}{2}+it}\left(\frac{P}{2\tau}- \frac{1}{2} \right) h_{\tau\tau}(t;\tau) \Biggr) \tanh(\pi t) t dt \\
		&= \frac{1}{4 \pi} \int_{-\infty}^{\infty}  \Xi_{\frac{1}{2}+it}\left( \frac{P}{2 \tau}- \frac{1}{2} \right) \left( h_{\tau\tau}(t;\tau)+\frac{1}{\tau^2} \left( \frac{1}{4}+t^2 \right) h(t;\tau) \right) \tanh(\pi t) t dt \\
		&=  \frac{4 \pi^2}{4 \pi} \int_{-\infty}^{\infty} \Xi_{\frac{1}{2}+it}\left( \frac{P}{2 \tau}- \frac{1}{2} \right)  h(t;\tau) \tanh(\pi t) t dt \\
		&= 4 \pi^2 \tau \Phi(\gamma),
		\label{eq:test-laplace}
	\end{align*}
	where we have used \eqref{eq:Xi-diff-eq} and \eqref{eq:h-diff-eq}. The result for $\tau<0$ follows immediately with the appropriate sign changes.
	
	It remains to show that $\Phi_{\infty} \in  V_{0,2\pi}$ as the decay condition follows from Corollary \ref{cor:M-infty-bounds}. Since $\Phi_{\infty}$ is not twice continuously differentiable everywhere this does not follow immediately. Instead, we employ the strategy from \cite[Lemma 6.5]{SupThetaI} and show that $\Phi_{\infty}$ is orthogonal to the spaces $V_{m,2\pi}$ for $m \neq 0$ by a limiting argument. Since $\Phi_{\infty} \in L^2(B_{\infty})$ by Corollary \ref{cor:M-infty-bounds}, it will follow from the discussion in \S\ref{sec:weil-rep} that $\Phi_{\infty} \in V_{0,2\pi}$. Let $\Phi' \in V_{m,2\pi}$, $m \neq 0$, be a Schwartz function satisfying the PDE \eqref{eq:isotypical-PDE} with $\omega = 2\pi$. Consider the domain
	$$
	V_R = \{ \gamma \in B_{\infty} \mid P(\gamma)-|\tau(\gamma)| \ge R^{-100}P(\gamma) \text{ and } P(\gamma) \le R \}.
	$$
It is compact and the function $\Phi_{\infty}$ is smooth in a neighbourhood of $V_R$. Let $L$ denote the operator $-\Delta+4\pi^2\det(\cdot)$. Then, by the continuity of $\Phi_{\infty}$, we have
\begin{equation}
	(m-0) \langle \Phi_{\infty},\Phi' \rangle = \lim_{R \to \infty} \left(\langle \Phi_{\infty}\restriction_{V_R} , L\Phi' \rangle-\langle L \Phi_{\infty}\restriction_{V_R}, \Phi' \rangle \right). 
	\label{eq:Minfty-M-int}
\end{equation}
	By Green's second identity/the divergence theorem, we have
	\begin{equation}
	\langle  \Phi_{\infty}\restriction_{V_R}, L\Phi' \rangle - \langle L \Phi_{\infty}\restriction_{V_R}, \Phi' \rangle  = \int_{\partial V_R} \left(\Phi_{\infty} (\nabla \overline{\Phi'})- (\nabla \Phi_{\infty}) \overline{\Phi'}   \right) \cdot \hat{\vect{n}} dA(\gamma),
	\label{eq:green-second-id}
	\end{equation}
	where $\nabla = \frac{1}{4}(\frac{\partial}{\partial a},-\frac{\partial}{\partial b},-\frac{\partial}{\partial c},\frac{\partial}{\partial d})^T$ in the coordinates $\gamma= [a,b,c]+d$. We shall show that the latter vanishes in the limit $R \to \infty$. For this, let $\delta =  R^{-100}$. The boundary $\partial V_R$ consists of the pieces
	$$\begin{aligned}
	E_R &= \{ \gamma \in B_{\infty} \mid P(\gamma)=R \text{ and } (1-\delta)R \ge |\tau(\gamma)| \}, \text{ and } \\
	F_R &= \{ \gamma \in B_{\infty} \mid P(\gamma) \le R \text{ and } P(\gamma)-|\tau(\gamma)| = \delta P(\gamma) \}.
	\end{aligned}$$
Using the bounds obtained in Corollary \ref{cor:M-infty-bounds}, we readily find
$$
\int_{E_R} \left(\Phi_{\infty} (\nabla \overline{\Phi'})- (\nabla \Phi_{\infty}) \overline{\Phi'}   \right) \cdot \hat{\vect{n}} dA(\gamma) \ll_{\epsilon} (R^{\frac{1}{2}}+ \delta^{-\frac{1}{2}-\epsilon} R^{\frac{3}{2}}) (1+R)^{-2023} R^{\frac{3}{2}} = o(1).
$$
The set $F_R$ is further split into the sets
$$
F_R(S) = \{ \gamma \in B_{\infty} \mid P(\gamma) \asymp S \text{ and } P(\gamma)-|\tau(\gamma)| = \delta P(\gamma) \}
$$
and
$$
F_{R,\eta} = \{ \gamma \in B_{\infty} \mid P(\gamma) \le \eta \text{ and } P(\gamma)-|\tau(\gamma)| = \delta P(\gamma) \},
$$
where $S$ runs dyadically between $R$ and $\eta= R^{-100}$. The surface area of $F_R(S)$ is $\ll \delta S^{\frac{3}{2}}$, hence
$$
\int_{F_R(S)} \left(\Phi_{\infty} (\nabla \overline{\Phi'})- (\nabla \Phi_{\infty}) \overline{\Phi'}   \right) \cdot \hat{\vect{n}} dA(\gamma) \ll_{\epsilon} (1+S^{\frac{1}{2}}+ \delta^{-\frac{1}{2}-\epsilon} S^{\frac{3}{2}}) (1+S)^{-2023} \delta S^{\frac{3}{2}} \ll R^{-1}.
$$
The surface area of $F_{R,\eta}$ is $\ll \eta^{\frac{3}{2}}$, hence
$$
\int_{F_{R,\eta}} \left(\Phi_{\infty} (\nabla \overline{\Phi'})- (\nabla \Phi_{\infty}) \overline{\Phi'}   \right) \cdot \hat{\vect{n}} dA(\gamma) \ll_{\epsilon} (1+\eta^{\frac{1}{2}}+ \delta^{-\frac{1}{2}-\epsilon} \eta^{\frac{3}{2}}) \eta^{\frac{3}{2}} = o(1).
$$
By combining these estimates and taking the limit $R \to \infty$, we conclude from \eqref{eq:Minfty-M-int}, \eqref{eq:green-second-id} that $\langle \Phi_{\infty},\Phi' \rangle = 0$ for $m \neq 0$, thus $\Phi_{\infty} \bot V_{m,2\pi}$ for $m \neq 0$ as the functions $\Phi'$ are dense, and consequently $\Phi_{\infty} \in \bigcap_{m \neq 0} V_{m,2\pi}^{\bot} =  V_{0,2\pi} \subset \Omega_{\infty}$.
\end{proof}

\section{Theta functions} \label{sec:theta}

In this Section, we collect the relvant machinery of \cite{SupThetaII}, in particular \S4, \S5, and the appendix thereof. 
Throughout, we let $h$ and $\Phi_{\infty}$ be as in Theorem \ref{thm:Existence-test-function}.

\subsection{Jacquet--Langlands lifts} Suppose $B$ is non-split ($\Leftrightarrow d_B >1$). For $\varphi \in \mathcal{F}$, we denote by $\varphi^{\JL}$ the arithmetically normalised newform in $L^2(\Gamma_0(d_BN) \backslash \mathbb{H})$ whose prime Hecke eigenvalues agree with those of $\varphi$ for almost all primes. The existence and uniqueness is guaranteed by the Jacquet--Langlands correspondence and strong multiplicity one.

In the case where $B= \Mat{2}{2}(\mathbb{Q})$ is split, we let $\varphi^{\JL}$ be the scalar multiple of $\varphi$ which is arithmetically normalised.

By the works of Iwaniec \cite{Iwaniec-Lower} and Hoffstein--Lockhart \cite{HoffLock}, we have
\begin{equation}
\|\varphi^{\JL}\|_2^2 = (d_BN(1+|t_{\varphi}|))^{o(1)} \cosh(\pi t_{\varphi})^{-1}.
\label{eq:L2-JL}
\end{equation}

\subsection{Theta function} 

For $z= x+iy \in \mathbb{H}$, $g \in \SL{2}(\mathbb{R})$, and $\ell \mid d_BN$, we let
\begin{equation}
\theta_{g,\ell}(z) = y \sum_{\gamma \in R(g;\ell)} \Phi_{\infty}\left(y^{\frac{1}{2}} \gamma\right) e(x\det(\gamma)).
\label{eq:theta-def}
\end{equation}
It is of moderate growth and is invariant under $\Gamma_0(d_BN)$. Furthermore, it is the classical counterpart of an adelic theta function $\Theta(g, \cdot)$. More precisely, we have
\begin{equation}
	\Theta(g,\tau_{\ell}s_{\infty}U_R^1)
	=
	\frac{\mu(\gcd(\ell,d_B))}{\ell}
	\theta_{g,\ell}( s_{\infty}i ),
	\label{eq:thetamaasstranform}
\end{equation}
for $\ell \div d_B N$ with $\mu$ the M{\"o}bius function, $\tau_{\ell}$ any matrix satisfying
\begin{equation}
	\tau_{\ell} \equiv \begin{pmatrix} 0 & 1 \\ -1 & 0 \end{pmatrix} \mod{\ell}, \quad \tau_{\ell} \equiv \begin{pmatrix} 1 & 0 \\ 0 & 1 \end{pmatrix} \mod{ d_BN / \ell},
	\label{eq:tau-ell}
\end{equation}
 and $U_R^1 = \sm \widehat{\mathbb{Z}} & \widehat{\mathbb{Z}} \\ d_BN \widehat{\mathbb{Z}} & \widehat{\mathbb{Z}} \esm \cap \SL{2}(\mathbb{A}_f)$, where $\widehat{\mathbb{Z}}$ is the completion of $\mathbb{Z}$ inside $\mathbb{A}_f$.

\subsection{Main identity}

For each $\ell \div d_BN$, we choose an element $\tau_{\ell} \in \SL{2}(\mathbb{Z})$ satisfying

\begin{prop} \label{prop:main-id}
	Let $\varphi$ be a Hecke--Maa{\ss} newform with spectral parameter $t_{\varphi}$ and $g \in \SL{2}(\mathbb{R})$. We have
	\begin{equation}\label{eq:thetamaassinner}
		\frac
		{
			\langle \Theta(g,\cdot), \varphi^{ \JL} \rangle
		}
		{
			\langle \varphi^{\JL}, \varphi^{\JL} \rangle
		}
		= h(t_{\varphi}) \cosh(\alpha t_{\varphi})
		\frac{|\varphi(gi)|^2}{V_{d_B,N}}.
	\end{equation}
\end{prop}

\begin{proof} By \cite[Appendix A]{SupThetaII}, it suffices to show
	$$
	\langle \Phi_{\infty} \restriction_{\SL{2}(\mathbb{R})} , \Xi_{\frac{1}{2}+it} \rangle_{\SL{2}(\mathbb{R})} = 2h(t) \cosh(\alpha t) K_{it}(2 \pi),
	$$
	which acts as a subsitute for the calculations in the proof of \cite[Prop.\@ A.14]{SupThetaII}. The former evaluates to
	$$
	4 \pi \int_0^{\infty} k(u;1) \Xi_{\frac{1}{2}+it}(u) du.
	$$
	Hence, the conclusion follows from \eqref{eq:selberg} and \eqref{eq:h-shape}.
\end{proof}

\subsection{$L^2$-bound}
\label{sec:theta-L2-bound}

\begin{prop} \label{prop:L2-bound} For $z_1, z_2 \in \Gamma \backslash \mathbb{H}$ and $g_j \in \SL{2}(\mathbb{R})$ with $g_j i = z_j$ for $j=1,2$, we have
	\begin{multline*}
		\frac{1}{V_{d_B,N}}	\sum_{\varphi} |h(t_{\varphi})|^2 \frac{\cosh(\alpha t_{\varphi})^2}{\cosh(\pi t_{\varphi})} \left( |\varphi(z_1)|^2-|\varphi(z_2)|^2 \right)^2 \\
		 \ll_{\epsilon} (d_BN)^{\epsilon} \sum_{i=1}^2 \sum_{\ell \div d_BN} \frac{1}{\ell} \int_{\frac{\sqrt{3}}{2} \frac{\ell^2}{d_BN}}^{\infty} \sum_{n \in \frac{1}{\ell}\mathbb{Z}} \left| \sum_{\substack{0 \neq \gamma \in R(\ell,g_i)\\\det(\gamma)=n}} \Phi_{\infty}\left(y^{\frac{1}{2}} \gamma\right) \right|^2 dy.
	\end{multline*}
\end{prop}

\begin{proof} By \eqref{eq:thetamaassinner} and Bessel's inequality, we have 
\begin{equation}
\frac{1}{V_{d_B,N}}	\sum_{\varphi} |h(t_{\varphi}) \cosh(\alpha t_{\varphi})|^2 \|\varphi^{\JL}\|_2^2 \left( |\varphi(z_1)|^2-|\varphi(z_2)|^2 \right)^2 \le V_{d_B,N} \| \Theta(g_1,\cdot) -\Theta(g_2, \cdot)  \|_2^2.
	\label{eq:Theta-Bessel}
\end{equation}
The latter is bounded by (see \cite[Lemma 4.1]{SupThetaII})
\begin{equation}
	\ll \frac{V_{d_B,N}}{V_{1,d_BN}} \sum_{\ell | d_BN} \int_{\frac{\sqrt{3}}{2} \frac{\ell^2}{d_BN}}^{\infty} \int_0^\ell \frac{1}{\ell^2} \left| \theta_{g_1,\ell}(z)- \theta_{g_2,\ell}(z) \right|^2 dx \frac{dy}{y^2}.
	\label{eq:Theta-Siegel}
\end{equation}
The inner integral is further equal to
\begin{equation}
	\begin{aligned}
		\int_0^\ell \frac{1}{\ell^2} \left| \theta_{g_1,\ell}(z)- \theta_{g_2,\ell}(z) \right|^2 dx &= \frac{1}{\ell} y^2 \sum_{n \in \frac{1}{\ell}\mathbb{Z}} \left| \sum_{i=1}^2(-1)^i \sum_{\substack{0 \neq \gamma \in R(\ell,g_i)\\\det(\gamma)=n}} \Phi_{\infty}\left(y^{\frac{1}{2}} \gamma\right) \right|^2 \\
		&\le \frac{2}{\ell} y^2 \sum_{i=1}^2 \sum_{n \in \frac{1}{\ell}\mathbb{Z}} \left| \sum_{\substack{0 \neq \gamma \in R(\ell,g_i)\\\det(\gamma)=n}} \Phi_{\infty}\left(y^{\frac{1}{2}} \gamma\right) \right|^2.
		\label{eq:theta-unipotent}
	\end{aligned}
\end{equation}

\end{proof}

\section{Long spectral window}
\label{sec:longwindow}

\subsection{Asymptotic notation}
\label{sec:asymptotic-notation}

In order to simplify the discussion, we shall introduce some asymptotic notation in the spirit of \cite[\S2.4]{WeylTriple}. We shall write
$$
A \prec B  \quad \Leftrightarrow \quad A \ll_{\epsilon} (d_BNT)^{\epsilon} B
$$
and
$$
A \precsim B  \quad \Leftrightarrow \quad A \ll_{\epsilon} (d_BNT)^{\epsilon} B+ \text{\emph{inconsequential}},
$$
where we call a term \emph{inconsequential} if it is $\ll_A (d_BNT)^{-A}$ for any $A \ge 0$.
We call a smooth function $\omega : \mathbb{R}^n_{\ge 0} \to \mathbb{C}$ \emph{flat} iff
$$
x_1^{j_1}\cdots x_n^{j_n} \omega^{(j_1,\cdots,j_n)}(\vect{x}) \precsim_{\vect{j}} 1,
$$
for every $\vect{j} \in \mathbb{N}_0^{n}$. Since we make quite liberal use of flat functions, we shall reserve the letter $\omega$ for them. In particular, we allow the flat functions to change in each occurrence. Flat functions serve two functions for us. Firstly, we may split an integral
\begin{multline}
\int_0^{\infty} f(x) dx = \int_0^{\infty} f(x) \omega_1\left(\frac{x}{X_1}\right) dx \\ + \sum_{\substack{X_1 \le X_2 \le X_3 \\ \text{dyadic}}} \int_0^{\infty} f(x) \omega_2\left(\frac{x}{X_2}\right) dx + \int_0^{\infty} f(x) \omega_3\left(\frac{x}{X_3}\right) dx,
\label{eq:integral-cut}
\end{multline}
where $\omega_1$ has support in $[0,1]$, $\omega_2$ has compact support in $]0,\infty[$, and $\omega_3$ has support in $[1,\infty[$, see \cite[\S5]{Harcos-shifted-convolution-Fourier-coef} for details. Secondly, for a compactly supported flat function $\omega$, we may write
\begin{equation}
\omega(\vect{x}) = \int \omega_1(x_1; s_1) \cdots \omega_n(x_n; s_n) F(\vect{s}) d\vect{s} + \text{\emph{inconsequential}},
\label{eq:decoupling}
\end{equation}
where for each $i$, $\omega_i(\,\cdot\,; s_1)$ is a uniform family of compactly supported flat functions and $\int |F(\vect{s})|d\vect{s} \precsim 1$, see \cite[\S2.4]{WeylTriple}. This allows us to decouple the variables at the expense of a negligible error and an extra integral which multiplicatively increases any bound by at most $\precsim 1$. Hence, we may safely omit this extra integral in our analysis.

\subsection{The test function}
For the fourth moment bound over the long window, we choose the test function $\Phi_{\infty}$ to have Selberg/Harish-Chandra transform $h(t;\tau)$ equal to
\begin{equation}
	h(t;\tau) = \cosh(\alpha t) \cdot 2 \sqrt{|\tau|} K_{it}(2 \pi |\tau|)
	\label{eq:h-long-choice}
\end{equation}
with the choice $\alpha = \frac{\pi}{2}- \frac{1}{T} < \frac{\pi}{2}$. The corresponding test function $\Phi_{\infty}$ is given by Theorem \ref{thm:Existence-test-function}. As we will consider the limit as $T \to \infty$, we shall assume that $\alpha \ge \frac{\pi}{4}$.

\subsection{The main proposition}

The remaining sections are dedicated to proving the following proposition.

\begin{prop} \label{prop:long-window-red-counting}
	For $g_1, g_2 \in \SL{2}(\mathbb{R})$ with $H(g_1), H(g_2) \ll (d_BNT)^A$ for some $A>0$ if $B$ is split and $T \ge 3$, there is a $g \in \{g_1, g_2\}$, $\ell \div d_BN$, $L \prec \frac{(d_BNT)^{\frac{1}{2}}}{\ell}$, $T^{-2} \ll \delta \le 1$, and $\delta^{\frac{1}{2}} T^{-1} \ll \heartsuit \ll \delta$ such that
	$$
	\frac{1}{d_BN}\sum_{|t_{\varphi}| \le T} \left(|\varphi(g_1 i)|^2-|\varphi(g_2 i)|^2  \right)^2 \precsim_A \frac{1}{\ell L^2 \heartsuit} \sum_{\substack{\gamma_1,\gamma_2 \in R(\ell;g)\\ \gamma_1,\gamma_2 \in  \Omega^\star(\delta,L) \cup \Psi^\star(\delta,L) \\ |\diamondsuit(\gamma_1)-\diamondsuit(\gamma_2)| \le \heartsuit L^4 \\ \det(\gamma_1)=\det(\gamma_2)}} 1.
	$$
\end{prop}
The proof is structured into parts. First, we dissect $\Phi_{\infty}$ into various parts and do a preliminary estimate on them. In a second step, we analyse more precisely how these parts contribute to the $L^2$-norm of the corresponding theta kernel.

\subsection{Preliminary estimates} \label{sec:test-estimate} 

In this section, we give preliminary estimates on $\Phi_{\infty}$ via the integral representation:
\begin{equation}
\Phi_{\infty}(\gamma) = -\frac{\sqrt{2}}{\pi} \int_0^{\infty} \frac{1}{\sqrt{v}} dQ(v+P(\gamma);\tau(\gamma)).
\tag{\ref{eq:M-def-2}}
\end{equation}
Note that in this case, in the definition of $Q$, see \eqref{eq:Q-P-tau-convolution}, $g(\ell)d\ell$ is just the Dirac measure at $\ell=0$, i.e.\@
\begin{equation}
Q(P;\tau) = \tfrac{1}{2} e^{-2\pi \cos(\alpha) P} \cos\left(2\pi \sin(\alpha) \sqrt{P^2-\tau^2}\right).
\label{eq:Q-long-window}
\end{equation}

In what follows, we shall cut the integral from $0$ to $\infty$ into smooth segments as in \eqref{eq:integral-cut} and proceed to bound these individually. There are three types of segments we consider:

\begin{enumerate}[label=(\Alph*)]
	\item \label{item:A} $0 \le v \le  \cos(\alpha)$, 
	\item \label{item:B} $v \asymp V$ with $\cos(\alpha) \le V \le  \cos(\alpha)^{-1}$, 
	\item \label{item:C} $\cos(\alpha)^{-1} \le v < \infty$.
\end{enumerate}





\subsubsection{Type \emph{\ref{item:A}}} \label{sec:type-A} We estimate trivially
\begin{multline}
	\int_0^{\infty} \frac{1}{\sqrt{v}} \omega\left(\frac{v}{\cos(\alpha)}\right) Q_P(v+P,\tau) dv \\
	\prec \int_0^{\cos(\alpha)} \frac{1}{\sqrt{v}} \left( \cos(\alpha)+ \frac{v+P }{1+\sqrt{P^2-\tau^2}}  \right) e^{-2\pi \cos(\alpha) (v+P)} dv \\
	\prec \cos(\alpha)^{\frac{1}{2}} \left( 1+ \min\{P,P\diamondsuit^{-\frac{1}{2}}\} \right) e^{-2\pi \cos(\alpha) P}.
	\label{eq:type-A}
\end{multline}

\subsubsection{Type \emph{\ref{item:B}}} \label{sec:type-B} In this case, we have 
\begin{multline}
	\int_0^{\infty} \frac{1}{\sqrt{v}} \frac{v+P}{\sqrt{(v+P)^2-\tau^2}} \sin(\alpha) \sin\left( 2\pi \sqrt{(v+P)^2-\tau^2} \sin(\alpha) \right) e^{-2\pi \cos(\alpha)(v+P)} \omega\left( \frac{v}{V} \right)dv\\
	+ \int_0^{\infty} \frac{1}{\sqrt{v}} \cos(\alpha) \cos\left(2\pi\sqrt{(v+P)^2-\tau^2}\sin(\alpha)\right)  e^{-2\pi \cos(\alpha)(v+P)} \omega\left( \frac{v}{V} \right)dv .
\label{eq:type-B456}
\end{multline}
for some smooth compactly supported flat functions $\omega$.
The latter integral is bounded by
\begin{equation}
\prec \cos(\alpha)^{\frac{1}{2}} e^{-2 \pi \cos(\alpha) P}.
\label{eq:type-B456-2nd}
\end{equation}
Alternatively, we may integrate by parts and arrive at
$$
\int_0^{\infty} \frac{1}{v^{\frac{3}{2}}} e^{-2\pi \cos(\alpha) (v+P)} \cos( 2\pi \sqrt{(v+P)^2- \tau^2} \sin(\alpha)) \omega\left( \frac{v}{V}\right) dv,
$$
where $\omega$ is some smooth compactly supported flat function. Integrating the oscillatory term by parts $A$-times yields the estimate
\begin{equation}
\prec_A V^{-\frac{1}{2}} e^{-2\pi \cos(\alpha) P} \left(V \cdot \frac{V+P}{V+V^{\frac{1}{2}}P^{\frac{1}{2}}+\diamondsuit^{\frac{1}{2}} }\right)^{-A},
\label{eq:type-B123}
\end{equation}
for any $A \ge 0$, where $\diamondsuit = P^2 -\tau^2 \ge 0$ in agreement with \eqref{eq:P-tau-diam-def}. We note
$$ V \cdot \frac{V+P}{V+V^{\frac{1}{2}}P^{\frac{1}{2}}+\diamondsuit^{\frac{1}{2}}} \asymp
\begin{cases}
	V , & V \ge P,\\
	V^{\frac{1}{2}}P^{\frac{1}{2}}, & \diamondsuit P^{-1} \le V \le P,\\
	V P \diamondsuit^{-\frac{1}{2}} , & V \le \diamondsuit P^{-1}.
\end{cases}
$$
Hence, we subdivide the segments of type \ref{item:B} into further sub-types. Let $\eta$ be another sufficiently and arbitrarily small but positive quantity which may depend on $\epsilon$ to allow for statements such as $(d_BNT)^{\eta} \prec 1$. Then, we distinguish the sub-types:
\begin{enumerate}[label=(B\arabic*)]
	\item \label{item:B1} $V \ge (d_BNT)^{\eta}$, 
	\item \label{item:B2} $V \ge P^{-1}(d_BNT)^{\eta}$ and $\diamondsuit P^{-1} (d_BNT)^{-\frac{\eta}{2}} \le V \le P$, 
	\item \label{item:B3} $V \ge \diamondsuit^{\frac{1}{2}}P^{-1}(d_BNT)^{\eta}$ and $V \le \diamondsuit P^{-1} (d_BNT)^{-\frac{\eta}{2}}$,
	\item \label{item:B4} $P \le V \le (d_BNT)^{\eta}$,
	\item \label{item:B5} $\diamondsuit P^{-1} (d_BNT)^{-\frac{\eta}{2}} \le V \le P$ and $V \le \min\{1,P^{-1}\} (d_BNT)^{\eta}$,
	\item \label{item:B6} $V \le \diamondsuit P^{-1} (d_BNT)^{-\frac{\eta}{2}}$ and $V \le \min\{1,\diamondsuit^{\frac{1}{2}}P^{-1}\}(d_BNT)^{\eta} $.
\end{enumerate}
For the segment types \ref{item:B1}, \ref{item:B2}, and \ref{item:B3}, we are going to use \eqref{eq:type-B123} which will lead to a negligible contribution. For the remaining ones, we shall use \eqref{eq:type-B456} instead.


\subsubsection{Type \emph{\ref{item:C}}} \label{sec:type-C} After integration by parts, we find
\begin{equation}
	\prec \cos(\alpha)^{\frac{1}{2}} e^{-2\pi \cos(\alpha) P} + \int_{\cos(\alpha)^{-1}}^{\infty} \frac{1}{v^{\frac{3}{2}}} e^{- 2\pi \cos(\alpha) (v+P)} dv \prec \cos(\alpha)^{\frac{1}{2}} e^{-2\pi \cos(\alpha) P}. 
	\label{eq:type-C}
\end{equation}

%

\subsection{$L^2$-bound of the theta kernel}
Our particular choice of test function $\Phi_{\infty}$ corresponding to $h(t;\tau)=\cosh(\alpha t) \cdot 2\sqrt{|\tau|}K_{it}(2\pi |\tau|)$ with $\alpha= \frac{\pi}{2}-\frac{1}{T}$ allows us to upper-bound the left-hand side of Proposition \ref{prop:long-window-red-counting} by the $L^2$-norm of the corresponding theta function from Section \ref{sec:theta}. In particular, Proposition \ref{prop:L2-bound} tells us that the left-hand side of Proposition \ref{prop:long-window-red-counting} is bounded by
\begin{equation}
	\prec \sum_{ \ell \div d_BN} \frac{1}{\ell} \int_{\frac{\sqrt{3}}{2} \frac{\ell^2}{d_BN}} \sum_{n \in \frac{1}{\ell} \mathbb{Z}} \left| \sum_{\substack{0 \neq \gamma \in R(\ell;g)\\ \det(\gamma)=n}} \Phi_{\infty}(y^{\frac{1}{2}}\gamma)   \right|^2  dy,
	\label{eq:L2-to-counting}
\end{equation}
which we continue to analyse. In a first step, we split $\Phi_{\infty}$ using Cauchy--Schwarz according to the segment types in \S\ref{sec:test-estimate}. This contributes a factor which is absorbed by the $\prec$ notation.

\subsubsection{Types \ref{item:A} and \ref{item:C}}
\label{sec:typeAC}
Let 
$$
Q(s,x) = \frac{1}{\Gamma(s)} \int_x^{\infty} t^s e^{-t} \frac{dt}{t}
$$
denote the normalised incomplete Gamma function. We open up the square in \eqref{eq:L2-to-counting} and pull the $y$-integral all the way in. We further employ the bounds \eqref{eq:type-A}, respectively \eqref{eq:type-C}, for the contributions of type \ref{item:A}, respectively \ref{item:C}, to $\Phi_{\infty}$. We thus find that the $y$-integral is of the shape
\begin{equation*}
\int_{Y}^{\infty} \cos(\alpha) \prod_{i=1,2}\min\left\{1+yP(\gamma_i),P(\gamma_i)\diamondsuit(\gamma_i)^{-\frac{1}{2}}\right\} e^{-2 \pi \cos(\alpha)y (P(\gamma_1)+P(\gamma_2))} dy,
\end{equation*}
which we further bound by the minimum of 
\begin{equation*}
	\sum_{i=1}^3 \frac{1}{\cos(\alpha)^{i-1}} \frac{1}{P(\gamma_1)+P(\gamma_2)} Q(i, 2 \pi \cos(\alpha)(P(\gamma_1)+P(\gamma_2))Y)
\end{equation*}
and
\begin{equation*}
	\frac{P(\gamma_1)\diamondsuit(\gamma_2)^{-\frac{1}{2}}P(\gamma_2)\diamondsuit(\gamma_2)^{-\frac{1}{2}}}{P(\gamma_1)+P(\gamma_2)} Q(1,2 \pi \cos(\alpha)(P(\gamma_1)+P(\gamma_2))Y ).
\end{equation*}
We note that if $P(\gamma) \asymp L^2$ and $\diamondsuit(\gamma) \ll \delta L^4$, then (up to some not relevant constants) we have $\gamma \in \Omega^{\star}(\delta,L) \cup \Psi^{\star}(\delta,L)$. We shall now dyadically partition $\sqrt{\max\{P(\gamma_1),P(\gamma_2)\}}$, which we denote by $L_j$, and $\diamondsuit(\gamma_i) P(\gamma_i)^{-2}$, which we denote by $\delta_a$. We let $\delta_a$ only run dyadically through the interval $[\cos(\alpha)^2,1]$. The $\gamma_i$'s with $\diamondsuit(\gamma_i) P(\gamma_i)^{-2} \le \cos(\alpha)^2$ shall all be collected into single set which is associated to $\delta_a \asymp \cos(\alpha)^2$. An application of Cauchy--Schwarz in the summation over $\delta_a$ and the super-polynomial decay of $Q(s,x)$ as soon as $x \gg s$ then shows that the contributions of type \ref{item:A} or \ref{item:C} to \eqref{eq:L2-to-counting} are bounded by
\begin{multline}
	\prec_A \sum_{\ell \div d_BN} \frac{1}{\ell} \sum_j \sum_a \frac{1}{L_j^2} \left(1+\frac{\ell^2}{d_BN} L_j^2 \cos(\alpha)  \right)^{-A} \\
	\times \min\left\{\delta_a^{-1},\cos(\alpha)^{-2}\right\}   \sum_{\substack{\gamma_1,\gamma_2 \in R(\ell;g)\\ \gamma_1,\gamma_2 \in \Omega^\star(\delta_a,L_j)\cup \Psi^\star(\delta_a,L_j)\\ \det(\gamma_1)=\det(\gamma_2)}} 1,
	\label{eq:AD-superpoly}
\end{multline}
for any $A \ge 0$, where $L_j$ runs dyadically through $]0,\infty[$ and $\delta_{a}$ runs dyadically through $[\cos(\alpha)^2,1]$. Note that the factor $|\log(\cos(\alpha))| \ll_{\epsilon} T^{\epsilon}$ from Cauchy--Schwarz has been absorbed by the $\prec$ notation. The super-polynomial decay and the polynomial bound on the counting from Propositions \ref{prop:Omega-lattice-counting} and \ref{prop:Psi-lattice-counting} allows us to truncate the sum at $L \prec \frac{(d_BNT)^{\frac{1}{2}}}{\ell}$ with an inconsequential error. Likewise, for $L_j \ll \min\{\ell^{-\frac{1}{2}}, \ell^{-1} H(g)^{-1} \}$ the inner most sum is empty, see Proposition \ref{prop:Omega-lattice-counting} and \ref{prop:Psi-lattice-counting}. Thus, the sum over $L_j$ may be truncated from the bottom, leaving $\prec 1$ summands. At this point, we may take the maximum and find that the bound is less or equal to the claimed upper bound in Proposition \ref{prop:long-window-red-counting} (take $\heartsuit \asymp \delta$).

\subsubsection{Types \ref{item:B1}-\ref{item:B3}} \label{sec:type-B123}

Here, we make use of \eqref{eq:type-B123}. After inserting this bound into the contribution of these segment types to \eqref{eq:L2-to-counting}, we forfeit any restriction imposed on $y$ and the $\gamma$'s imposed by \ref{item:B1}-\ref{item:B3}. This leads to the upper bound
$$
\prec_A (d_BNT)^{-A} \sum_{ \ell \div d_BN} \frac{1}{\ell} \int_{\frac{\sqrt{3}}{2} \frac{\ell^2}{d_BN}} \sum_{n \in \frac{1}{\ell} \mathbb{Z}}  \sum_{\substack{0 \neq \gamma_1,\gamma_2 \in R(\ell;g)\\ \det(\gamma)=n}} V^{-1} e^{-2\pi \cos(\alpha) (P(\gamma_1)+P(\gamma_2))}  dy.
$$
This may be bounded analogous to the previous section. The resulting bound is:
\begin{multline}
	\prec_A (d_BNT)^{-A} \sum_{\ell \div d_BN} \frac{1}{\ell} \sum_j \frac{1}{L_j^2} \left(1+\frac{\ell^2}{d_BN} L_j^2 \cos(\alpha)  \right)^{-A} \\
	 \times \frac{1}{\cos(\alpha)V} \sum_{\substack{\gamma_1,\gamma_2 \in R(\ell;g) \cap \Omega^\star(1,L_j)\\ \det(\gamma_1)=\det(\gamma_2)}} 1,
	\label{eq:C-superpoly}
\end{multline}
where $V \ge \cos(\alpha) \gg \frac{1}{T}$. Hence, by Proposition \ref{prop:Omega-lattice-counting}, this whole expression is inconsequential.

%

\subsubsection{Types \ref{item:B4}-\ref{item:B6}}

These are the most difficult cases which, in particular, require a more stringent type of counting that needn't be and hasn't been considered in the prequels. We note that the second integral in \eqref{eq:type-B456} is also bounded by $\cos(\alpha)^{\frac{1}{2}} e^{-2\pi \cos(\alpha) P}$ and thus satisfies the same bound as type \ref{item:C}. Hence, we turn our attention to the first integral of \eqref{eq:type-B456}. In a first step, we make the substitution $v=uy$.

We shall now consider the contribution to \eqref{eq:L2-to-counting}. We open up the square once more and abbreviate $P(\gamma_i)$ by $P_i$, $\diamondsuit(\gamma_i)$ by $\diamondsuit_i$, and $\tau=\tau_i=\tau(\gamma_i)$ for $i=1,2$. The total contribution of these types are thus given by
\begin{multline}
 \sum_{\ell \div d_BN} \frac{1}{\ell} \sum_{n \in \frac{1}{\ell} \mathbb{Z}} \sum_{\substack{0 \neq \gamma_1, \gamma_2 \in R(\ell;g) \\ \det(\gamma_1)=\det(\gamma_2)=n}} \int_{\frac{\ell^2}{2d_BN}}^{\infty} y^2 \prod_{i=1}^2 \Biggl( y^{\frac{1}{2}} \int_0^{\infty} \frac{1}{\sqrt{u}} \frac{u+P_i}{\sqrt{(u+P_i)^2-\tau^2}} \\ \times \sin\left( 2\pi y \sqrt{(u+P_i)^2-\tau^2} \sin(\alpha) \right) e^{-2\pi \cos(\alpha) y (u+P_i)} \omega\left(\frac{yu}{V}\right) du \Biggr)  \frac{dy}{y^2},
\label{eq:TypeBwhole}
\end{multline}
subject to the additional restrictions imposed by \ref{item:B4}-\ref{item:B6}, which now read
\begin{enumerate}[label=(B\arabic*${}^{\ast}$)] \setcounter{enumi}{3}
	\item \label{item:B4*} $yP_i \le V \le (d_BNT)^{\eta}$,
	\item \label{item:B5*} $y\diamondsuit_i P_i^{-1} (d_BNT)^{-\frac{\eta}{2}} \le V \le yP_i$ and $V \le \min\{1,y^{-1}P_i^{-1}\} (d_BNT)^{\eta}$,
	\item \label{item:B6*} $V \le y\diamondsuit_i P_i^{-1} (d_BNT)^{-\frac{\eta}{2}}$ and $V \le \min\{1,\diamondsuit^{\frac{1}{2}}_iP_i^{-1}\}(d_BNT)^{\eta} $.
\end{enumerate}
First, we shall deal with the tail of the $y$-integral. For this, we bound the $u$-integrals trivially:
\begin{multline*}
\int_0^{\infty} \frac{1}{\sqrt{u}} \frac{u+P_i}{\sqrt{(u+P_i)^2-\tau^2}} \sin\left( 2 \pi y \sqrt{(u+P_i)^2 - \tau^2} \sin(\alpha) \right) e^{-2\pi \cos(\alpha)y (u+P_i)} \omega\left(\frac{yu}{V}\right) du \\
\prec \left(\frac{V^{1/2}}{y^{1/2}} + P_i^{1/2} \right) e^{-2 \pi \cos(\alpha) y P_i} \prec \left(y^{-1/2} + P_i^{1/2} \right) e^{-2 \pi \cos(\alpha) y P_i},
\end{multline*}
where we recall that $V \prec 1$ from \ref{item:B4*}-\ref{item:B6*}. Hence, the contribution from $y \ge Y$ to the $y$-integral is bounded by
\begin{multline*}
	\int_Y^{\infty} y \left( \frac{1}{y}+P_1+P_2 \right) e^{-2\pi \cos(\alpha) y (P_1+P_2) } dy \\
	\prec \frac{1}{\cos(\alpha)(P_1+P_2)} Q(1, 2\pi Y \cos(\alpha) (P_1+P_2)) \\
	+ \frac{1}{\cos(\alpha)^2(P_1+P_2)} Q(2, 2\pi Y \cos(\alpha)(P_1+P_2)).
\end{multline*}
By dyadically partitioning $\sqrt{\max\{P_1,P_2\}}$, which we denote by $L_j$, and using the super-polynomial decay of the incomplete Gamma function as in the previous cases, we see that contribution from the tail $y \ge Y$ is bounded by
$$
 \sum_{\ell \div d_BN} \frac{1}{\ell} \sum_j \frac{1}{\cos(\alpha)^2 L_j^2} \left(1+YL_j^2 \cos(\alpha) \right)^{-A} \sum_{\substack{\gamma_1,\gamma_2 \in R(\ell;g) \cap \Omega^\star(1,L_j)\\ \det(\gamma_1)=\det(\gamma_2)}} 1.
$$
We note that for $0 \neq \gamma \in R(\ell;g)$, we have $P(\gamma) \gg \min\{\ell^{-1}, \ell^{-2}H(g)^{-2}\}$, see Proposition \ref{prop:Omega-lattice-counting}. Hence, if we let $Y \ge T \max\{\ell, \ell^2 H(g)^2\} (d_BNT)^{\eta}$ this will turn out to be inconsequential after once more referring to Proposition \ref{prop:Omega-lattice-counting}.

The remaining $y$-integral we partition into smooth dyadic intervals $y \asymp Y$ with
\begin{equation} \label{eq:Ydelimitation}
\frac{\ell^2}{d_BN} \ll Y \le T  \max\{\ell, \ell^2  H(g)^2 \} (d_BNT)^{\eta}.
\end{equation}
By the discussion in Section \ref{sec:asymptotic-notation}, we may decouple the variables and replace
$$
\omega\left(\frac{y}{Y}\right)\omega\left(\frac{yu_1}{V}\right)\omega\left(\frac{yu_2}{V}\right) \approx \omega\left(\frac{y}{Y}\right) \omega\left(\frac{u_1}{U}\right) \omega\left(\frac{u_2}{U}\right), \quad \text{where } U=V/Y
$$
and every instance of the flat function $\omega$ may differ, up to a inconsequential error (and an omitted integral). This (inconsequential) error then contributes at most
\begin{equation}
	\prec_A (d_BNT)^{-A}  \sum_{\ell \div d_BN} \sum_j Y^2(Y^{-1}+ L_j^2) (1+\cos(\alpha)YL_j^2)^{-A} \sum_{\substack{\gamma_1,\gamma_2 \in R(\ell;g) \cap \Omega^\star(1,L_j)\\ \det(\gamma_1)=\det(\gamma_2)}} 1
	\label{eq:TypeB-y-tail-contribution}
\end{equation}
to \eqref{eq:TypeBwhole}. Proposition \ref{prop:Omega-lattice-counting} shows that is inconsequential again.

For the newly decoupled variables, we may pull the $y$-integral to the inside. We are thus led to study
\begin{multline}
\int_0^{\infty} y  \sin\left(2\pi y \sqrt{(u_1+P_1)^2-\tau^2} \sin(\alpha)\right) \sin\left(2\pi y \sqrt{(u_2+P_2)^2-\tau^2} \sin(\alpha)\right) \\ e^{-2\pi y \cos(\alpha) (u_1+P_1+u_2+P_2)} \omega\left(\frac{y}{Y}\right) dy.
\label{eq:TypeB-y-pulled-in}
\end{multline}
Let $u_1,u_2 \asymp U = V/Y$. Assuming
\begin{equation}
Y \left( \frac{1}{T}(2U+P_1+P_2)+\left| \sqrt{(u_1+P_1)^2-\tau^2}- \sqrt{(u_2+P_2)^2-\tau^2}   \right|\right) \ge (d_BNT)^{\eta},
\label{eq:Yfirstder}
\end{equation}
then integration by parts shows that \eqref{eq:TypeB-y-pulled-in} is $\prec_A (d_BNT)^{-A}Y^2$, for any $A \ge 0$. This leads again to a contribution of \eqref{eq:TypeB-y-tail-contribution} to \eqref{eq:TypeBwhole}, which is inconsequential. Thus, we assume from now on that
\begin{equation}
	Y \left( \frac{1}{T}(2U+P_1+P_2)+\left| \sqrt{(u_1+P_1)^2-\tau^2}- \sqrt{(u_2+P_2)^2-\tau^2}   \right|\right) \le (d_BNT)^{\eta}
	\label{eq:Yrestriction}
\end{equation}
holds. We also record the bound
\begin{multline} \label{eq:u-y-int-bound}
	\int_{\frac{\ell^2}{2d_BN}}^{\infty} y^2 \prod_{i=1}^2 \Biggl( y^{\frac{1}{2}} \int_0^{\infty} \frac{1}{\sqrt{u}} \frac{u+P_i}{\sqrt{(u+P_i)^2-\tau^2}} \\ \times \sin\left( 2\pi y \sqrt{(u+P_i)^2-\tau^2} \sin(\alpha) \right)  e^{-2\pi \cos(\alpha) y (u+P_i)} \omega\left(\frac{u}{U}\right) du \Biggr) \omega\left( \frac{y}{Y} \right)  \frac{dy}{y^2} \\
	 \ll Y^2 U \prod_{i=1}^2 \frac{U+P_i}{U+U^{\frac{1}{2}}P_i^{\frac{1}{2}}+\diamondsuit_i^{\frac{1}{2}}}.
\end{multline}

\paragraph{Type \ref{item:B4*}: $YP_i \ll YU \prec 1$} Here, we know that $\frac{1}{T} \asymp \cos(\alpha) \ll V = YU$ from \ref{item:B} and $Y(P_1+P_2) \prec T$ from \eqref{eq:Yrestriction}. We dyadically partition $\sqrt{\max\{P_1,P_2\}}$, which we denote by $L_j$. From the estimate \eqref{eq:u-y-int-bound}, it follows that the contribution of this type to \eqref{eq:TypeBwhole} is
$$
\prec \sum_{\ell \div d_BN} \frac{1}{\ell} \sum_j Y^2 U \sum_{\substack{\gamma_1,\gamma_2 \in R(\ell;g) \cap \Omega^\star(1,L_j) \\ \det(\gamma_1)=\det(\gamma_2)}} 1
$$
summed over all dyadic ranges \eqref{eq:Ydelimitation}. Since $L^2_j \ll P_1+P_2 \prec U$, we have $Y^2U \prec L^{-2}$. Hence, this last expression is less or equal to the claimed upper bound in Proposition \ref{prop:long-window-red-counting}.


\paragraph{Type \ref{item:B5*}: $\diamondsuit_i P_i^{-1}  \prec U \ll P_i$ and $YU \prec \min\{1,Y^{-1}(P_1+P_2)^{-1}\}$} Here, we know $\frac{1}{T} \asymp \cos(\alpha) \ll YU$ from \ref{item:B} as well as $Y(P_1+P_2) \prec T$ from \eqref{eq:Yrestriction}. We dyadically partition $\sqrt{\max\{P_1,P_2\}}$, which we denote by $L_j$. Similarly, we dyadically partition $\diamondsuit_i/P_i^2$, which we denote by $\delta_a$, and apply Cauchy--Schwarz as we did for the types \ref{item:A} and \ref{item:C}. From the estimate \eqref{eq:u-y-int-bound}, it follows that the contribution of this type to \eqref{eq:TypeBwhole} is
$$
\prec \sum_{\ell \div d_BN} \frac{1}{\ell} \sum_j \sum_a Y^2 L_j^2 \sum_{\substack{\gamma_1,\gamma_2 \in R(\ell;g)\\ \gamma_1,\gamma_2 \in  \Omega^\star(\delta_a,L_j) \cup \Psi^\star(\delta_a,L_j) \\ \det(\gamma_1)=\det(\gamma_2)}} 1
$$
summed over all dyadic ranges \eqref{eq:Ydelimitation}. We have $\delta_aL_j^2 \prec U$ and $Y^2UL_j^2 \prec 1$. We also have $Y^2L_j^2 \prec T^2L_j^{-2}$. Hence, $Y^2L_j^2 \prec L_j^{-2}\min\{\delta_a^{-1},T^2\}$. Thus, this last expression is less or equal to the claimed upper bound in Proposition \ref{prop:long-window-red-counting}.


\paragraph{Type \ref{item:B6*}: $U \le  \diamondsuit_i P_i^{-1} (d_BNT)^{-\frac{\eta}{2}}$ and $YU \prec \min\{1, \diamondsuit_i^{\frac{1}{2}}P_i^{-1}\}$} Here, we have
$$
\sqrt{(u+P_i)^2-\tau^2} = (1+o(1)) \diamondsuit_i^{\frac{1}{2}}.
$$
From \ref{item:B}, we know $\frac{1}{T} \asymp \cos(\alpha) \ll YU$ and, from \eqref{eq:Yrestriction}, we know $YU \prec T$, $YP_i \prec T$, and $Y|\diamondsuit_1^{\frac{1}{2}}-\diamondsuit_2^{\frac{1}{2}}| \prec 1$. The latter we may rewrite as
$$
Y(P_1+P_2) \cdot \frac{|\diamondsuit_1-\diamondsuit_2|}{(P_1+P_2)^2} \cdot \left( \frac{\diamondsuit_1^{\frac{1}{2}}+\diamondsuit_2^{\frac{1}{2}}}{P_1+P_2} \right)^{-1} = Y|\diamondsuit_1^{\frac{1}{2}}-\diamondsuit_2^{\frac{1}{2}}| \prec 1.
$$
We dyadically partition $\sqrt{\max\{P_1,P_2\}}$ by $L_j$. Similarly, we dyadically partition $\diamondsuit_i/P_i^2$ by $\delta_a$ and apply Cauchy--Schwarz as we did for the types \ref{item:A} and \ref{item:C}. Lastly, we dyadically partition $|\diamondsuit_1-\diamondsuit_2|/(P_1+P_2)^2$ by $\heartsuit_b$. From the estimate \eqref{eq:u-y-int-bound}, it follows that the contribution of this type to \eqref{eq:TypeBwhole} is
$$
\prec \sum_{\ell \div d_BN} \frac{1}{\ell} \sum_j \sum_a \sum_b Y^2 \left( U^{\frac{1}{2}} \delta_a^{-\frac{1}{2}}   \right)^2  \sum_{\substack{\gamma_1,\gamma_2 \in R(\ell;g)\\ \gamma_1,\gamma_2 \in  \Omega^\star(\delta_a,L_j) \cup \Psi^\star(\delta_a,L_j) \\ |\diamondsuit_1-\diamondsuit_2| \le \heartsuit_b L_j^4 \\ \det(\gamma_1)=\det(\gamma_2)}} 1.
$$
summed over all dyadic ranges \eqref{eq:Ydelimitation}. We have
$$\begin{aligned}
 Y^2U\delta_a^{-1} & \ll Y^2 (\delta_a L_j^2) \delta_a^{-1} = L_j^{-2} (YL_j^2)^2 \prec L_j^{-2} T^2, \\
 Y^2U\delta_a^{-1} & = L_j^{-2}(YU\delta_a^{-\frac{1}{2}}) \delta_a^{-\frac{1}{2}} (YL_j^2)  \prec  L_j^{-2} \delta_a^{-\frac{1}{2}} T, \\
 Y^2U\delta_a^{-1} & = L_j^{-2} (Y L_j^2 \delta_a^{-\frac{1}{2}}) (YU\delta_a^{-\frac{1}{2}}) \prec L_j^{-2} \heartsuit_b^{-1} .
\end{aligned}$$
Thus, this last expression is less or equal to the claimed upper bound in Proposition \ref{prop:long-window-red-counting}.

\subsection{Proof of main theorems} \label{sec:proof-main-thms} Theorems \ref{thm:fourth-hybrid unconditional}-\ref{thm:Eichler-order-main-thm} follow immediately from Proposition \ref{prop:long-window-red-counting} and the counting propositions and conjecture: Proposistions \ref{prop:Omega-lattice-counting}, \ref{prop:Psi-lattice-counting}, and Conjecture \ref{conj:heart-lattice-counting}. For the individual bound in Theorem \ref{thm:Eichler-order-main-thm}, we not that we may find a point $w \in \Gamma \backslash \mathbb{H}$ such that $|\varphi(w)| \le 1$ as $\|\varphi\|_2=1$ and the measure is a probability measure. It thus remains to prove Theorem \ref{thm:hybrid-indiv}. We require the following lemma.

\begin{lem}[{\cite[Prop.\@ 3.1]{Thybrid}}] \label{lem:split-y-large} 
	For an $L^2$-normalised Hecke--Maa{\ss} form $\varphi$ on $\Gamma_0(N) \backslash \mathbb{H}$ and a point $z \in \mathbb{H}$, we have
	$$
	|\varphi(z)| \ll_{\epsilon} (N (1+|\lambda_{\varphi}|))^{\epsilon}  \left( \frac{(1+|\lambda_{\varphi}|)^{\frac{1}{4}}}{H(z)^{\frac{1}{2}}}+ (1+|\lambda_{\varphi}|)^{\frac{1}{12}} \right).
	$$
\end{lem}
Thus, Theorem \ref{thm:hybrid-indiv} follows from Lemma \ref{lem:split-y-large} if $H(z) \ge (1+|\lambda_{\varphi}|)^{\frac{1}{4}} N^{-\frac{1}{2}}$. Assume the contrary and take $w \in \mathbb{H}$ arbitrarily such that $H(w) = (1+|\lambda_{\varphi}|)^{\frac{1}{4}} N^{-\frac{1}{2}}$ and thus,
$$
|\varphi(w)| \ll_{\epsilon} N^{\frac{1}{4}+\epsilon} (1+|\lambda_{\varphi}|)^{\frac{3}{8}+\epsilon}.
$$
The conclusion of Theorem \ref{thm:hybrid-indiv} thus follows from Theorem \ref{thm:fourth-hybrid unconditional} and the triangle inequality.

\appendix
\section{Bounds for special functions}
\label{app:bounds-special}

In this section, we collect some bounds for special functions that are being used for the absolut convergence of certain integrals.

\begin{lem} \label{lem:h(t,tau)-bounds} We have for $t \in \mathbb{R}$ and $x > 0$ that
	\begin{align*}
		\frac{\partial^j}{\partial x^j} K_{it}(x) & \ll (1+|\log(x)|) \left(\tfrac{1+|t|}{x}\right)^j e^{-\frac{\pi}{2}|t|}, \quad j=0,1,2.
	\end{align*}
	Thus, for $h(t;\tau)$ as in Theorem \ref{thm:Existence-test-function} with $t,\tau \in \mathbb{R}$ and $\tau \neq 0$, we have that 
	\begin{align*}
		\frac{\partial^j}{\partial \tau^j}  h(t;\tau) & \ll 
		(1+|\log(|\tau|)|) |\tau|^{\frac{1}{2}} \left(\tfrac{1+|t|}{|\tau|}\right)^j e^{-(\frac{\pi}{2}-\alpha)|t|}  , \quad j=0,1,2.
	\end{align*}
\end{lem}
\begin{proof} The bounds on $h(t; \tau)$ and its derivatives will follow from the bounds on the $K$-Bessel function and the assumptions on $h(t)$ in Theorem \ref{thm:Existence-test-function}. The bound on $K_{it}(x)$ is recorded in \cite[Eq.\@ (3.58)]{thesis}. We apply the same strategy here and split into several ranges of the parameters.	
	\paragraph{Range $|t| \le 2(1+x)$} 	
	The $K$-Bessel function may be represented by the integral representation (see \cite[\S 6.16]{ToBF})
	\begin{equation}
		K_{it}(x) = \frac{\Gamma(\frac{1}{2}+it) (2x)^{it}}{\sqrt{\pi}} \int_0^{\infty} \frac{\cos(y)}{(y^2+x^2)^{\frac{1}{2}+it}} dy.
		\label{eq:K-Bessel-int-rep}
	\end{equation}
	By integration by parts, we split the integral into
	\begin{multline}
		\int_0^{\infty} \frac{\cos(y)}{(y^2+x^2)^{\frac{1}{2}+it}} dy = \int_0^{A} \frac{\cos(y)}{(y^2+x^2)^{\frac{1}{2}+it}} dy\\
		- \frac{\sin(A)}{(A^2 + x^2)^{\frac{1}{2}+it}} + (1+2it)\int_{A}^{\infty} \frac{x\sin(y)}{(y^2+x^2)^{\frac{3}{2}+it}} dy.
		\label{eq:K-int-split}\end{multline}
	By Lebesgue dominated convergence, we may pull derivatives in $x$ on the right-hand side to the inside and the put the pieces back together to arrive at 
	$$
	\frac{\partial^j}{\partial x^j} \int_0^{\infty} \frac{\cos(y)}{(y^2+x^2)^{\frac{1}{2}+it}} dy = \int_0^{\infty} \frac{\partial^j}{\partial x^j}  \left(\frac{\cos(y)}{(y^2+x^2)^{\frac{1}{2}+it}} \right) dy, \quad j=1,2.
	$$
	By bounding absolutely, we conclude
	$$
	\frac{\partial^j}{\partial x^j} \int_0^{\infty} \frac{\cos(y)}{(y^2+x^2)^{\frac{1}{2}+it}} dy \ll x^{-j} (1+|t|)^j, \quad j=1,2.
	$$
	Similarly, we find from \eqref{eq:K-int-split} with $A=x+1$ that
	$$
	\int_0^{\infty} \frac{\cos(y)}{(y^2+x^2)^{\frac{1}{2}+it}} dy \ll \left(1+ |\log(x)|\right).
	$$
	By inserting these bounds into (derivates of) \eqref{eq:K-Bessel-int-rep}, we find
	$$
	\frac{\partial^j}{\partial x^j} K_{it}(x) \ll (1+|\log(x)|) x^{-j} (1+|t|)^j  e^{-\frac{\pi}{2}|t|}, \quad j=1,2.
	$$

	\paragraph{Range $|t| \ge 2x \ge 2$}
	
	The bounds derived in \cite[Prop.\@ 2]{Kbesselbound} read
	$$
	\frac{\partial^j}{\partial x^j} K_{it}(x) \ll_{\epsilon} x^{-j}|t|^{j-\frac{1}{2}} e^{-\frac{\pi}{2}|t|}
	$$
	for this range, where $j=0,1,2$.

	\paragraph{Range $x \le 1$, $|t| \ge 2$}
	In this range, we make use of the power series expansion of $K_{it}(x)$:
	$$
	K_{it}(x) =  \frac{\pi}{2\sinh(\pi t)}  \sum_{m=0}^{\infty} \left(\frac{(x/2)^{2m-it}}{m! \Gamma(1-it+m)}  - \frac{(x/2)^{2m+it}}{m! \Gamma(1+it+m)} \right).
	$$
	By bounding absolutely, this gives the bound
	$$
	K_{it}(x) \ll |t|^{-\frac{1}{2}} e^{-\frac{\pi}{2}|t|} \exp\left( \frac{x^2}{4} \right) \ll |t|^{-\frac{1}{2}} e^{-\frac{\pi}{2}|t|}.
	$$
	For the derivatives, we similarly find
	$$
	\frac{\partial^j}{\partial x^j} K_{it}(x) \ll  x^{-j} |t|^{j-\frac{1}{2}} e^{-\frac{\pi}{2}|t|}, \quad j=1,2.
	$$
\end{proof}

\begin{lem} \label{lem:Xi-derivative-bounds} We have
	\begin{align*}
		\Xi_{\frac{1}{2}+it}(u) & \ll \begin{cases} 1, & u \le 2, \\ \log(u) u^{-\frac{1}{2}}, & u \ge 2, \end{cases}\\
		\Xi_{\frac{1}{2}+it}'(u) & \ll (1+|t|^2) \begin{cases} 1 , & u \le 2, \\  \log(u) u^{-\frac{3}{2}}, & u \ge 2, \end{cases}\\
		\Xi_{\frac{1}{2}+it}''(u) & \ll (1+|t|^2) \begin{cases} u^{-1} , & u \le 2, \\ \log(u)u^{-\frac{5}{2}} , & u \ge 2. \end{cases}
	\end{align*}
\end{lem}
\begin{proof} We make use of the explicit description of $\Xi_s(u)= P_{-s}(2u+1)$ in terms of the associated Legendre function. Mehler's integral representation (see 
	\cite[\S 8.715 Eq.\@ 1]{ToI}) for the associated Legendre function is given by
	\begin{equation}\begin{aligned}
			P_{-\frac{1}{2}+it}^{\mu}(\cosh(\vartheta)) &= \sqrt{\frac{2}{\pi}}\frac{ (\sinh(\vartheta))^{\mu}}{\Gamma(\frac{1}{2}-\mu)} \int_0^{\vartheta} \frac{\cos(\alpha t)}{(\cosh(\vartheta)-\cosh(\alpha))^{\frac{1}{2}+\mu}} d\alpha \\
			&= \frac{1}{\sqrt{\pi}} \frac{(\sinh(\vartheta))^{\mu}}{\Gamma(\frac{1}{2}-\mu)} \int_0^{\vartheta} \frac{\cos(\alpha t)}{\left(\sinh\left(\frac{\vartheta-\alpha}{2}\right)\sinh\left(\frac{\vartheta+\alpha}{2}\right)\right)^{\frac{1}{2}+\mu}} d\alpha,
			\label{eq:Legendre-Mehler}
	\end{aligned}\end{equation}
	for $\mu < \frac{1}{2}$. Suppose first that $0 \le \vartheta \le 1$. Then,
	$$
	|P^{\mu}_{-\frac{1}{2}+it}(\cosh(\vartheta))| \ll \frac{\vartheta^{\mu}}{\Gamma(\frac{1}{2}-\mu)} \int_0^{\vartheta} \frac{1}{(\vartheta^2-\alpha^2)^{\frac{1}{2}+\mu}} d\alpha \ll \frac{\vartheta^{-\mu}}{\Gamma(\frac{3}{2}-\mu)}.
	$$
	If $\vartheta \ge 1$, we split the integral in to $[0, \vartheta-1]$ and $[\vartheta-1,\vartheta]$. We have
	$$\begin{aligned}
		|P^{\mu}_{-\frac{1}{2}+it}(\cosh(\vartheta))| &\ll \frac{e^{\mu \vartheta}}{\Gamma(\frac{1}{2}-\mu)} \left[ \int_0^{\vartheta-1} \frac{1}{e^{(\frac{1}{2}+\mu)\vartheta}} d\alpha + \int_{0}^{1} \frac{1}{((1-\alpha)e^{\vartheta})^{\frac{1}{2}+\mu}} d\alpha \right] \\
		&\ll \frac{1}{\Gamma(\frac{1}{2}-\mu)} \vartheta e^{-\frac{\vartheta}{2}}+\frac{1}{\Gamma(\frac{3}{2}-\mu)}e^{-\frac{\vartheta}{2}}.
	\end{aligned}$$
	This gives the desired bound for $\Xi_{-\frac{1}{2}+it}(u)$. For the first derivative of $\Xi_{-\frac{1}{2}+it}(u)$, we make use of the relation \cite[\S 8.731 Eq.\@ 1(2)]{ToI} to write
	$$
	\sinh(\vartheta) (P^{0}_{-\frac{1}{2}+it})^{(1)	}(\cosh(\vartheta)) = -(\tfrac{1}{4}+t^2)  P^{-1}_{-\frac{1}{2}+it}(\cosh(\vartheta))
	$$
	and for the second derivative, we make use of Equation \eqref{eq:Xi-diff-eq}.


\end{proof}

\bibliography{RafBib}
\end{document}